\documentclass[11pt,twoside,reqno]{amsart}

\usepackage{amsmath,amsfonts,amsthm,amssymb}
\usepackage{graphicx,psfrag,overpic}
\usepackage{subfigure}
\usepackage{caption2}
\usepackage{cite}
\usepackage{bm,enumerate}
\usepackage{cases}
\usepackage[all,cmtip,line]{xy}
\usepackage{array,CJK}
\usepackage{color}
\usepackage{booktabs}
\usepackage{underscore}
\usepackage{color}
\usepackage{appendix}
\usepackage{makecell}
\usepackage{pgfplots}
\usepackage{tikz}
\usepackage{pstricks,pst-node}

\numberwithin{equation}{section}
\numberwithin{figure}{section}
\numberwithin{table}{section}

\usepackage{hyperref}
\usepackage{cleveref}

\usepackage{graphics}
\usepackage{epsfig}
\newtheorem{claim}{\bf \t}[part]

\newtheorem{theorem}{Theorem}[section]
\newtheorem{corollary}[theorem]{Corollary}

\newtheorem{lemma}{Lemma}[section]
\newtheorem{proposition}[theorem]{Proposition}
\newtheorem{remark}{Remark}[section]

\def\vep{\varepsilon}
\def\ep{\epsilon}
\def\pt{\partial}

\def\div{\mathrm{div}}

\def\tC{{\widetilde{C}}}

\def\ttau{\widetilde{\tau}}
\def\fT{\mathcal{T}}
\def\tfT{\widetilde{\fT}}
\def\Nhat{\widehat{N}}

\def\trho{{\widetilde{\rho}}}

\def\tu{\widetilde{u}}
\def\tbfu{\widetilde{\bfu}}

\def\mup{\mu^\prime}

\def\Ma{\mathrm{Ma}}
\def\Re{\mathrm{Re}}
\def\Pr{\mathrm{Pr}}

\def\bfx{\mathbf{x}}
\def\bfu{\mathbf{u}}

\def\bfpsi{\bm{\psi}}
\def\Du{\mathfrak{D}(\bfu)}
\def\Dtbfu{\mathfrak{D}(\tbfu)}
\def\Dpsi{\mathfrak{D}(\bfpsi)}
\def\bfW{\mathbf{W}}
\def\tbfW{\widetilde{\bfW}}
\def\md{\mathrm{d}}

\def\bfF{\mathbf{F}}
\def\bff{\mathbf{f}}
\def\bfM{\mathbf{M}}
\def\bfR{\mathbf{R}}
\def\tfH{\mathcal{H}}

\def\R{\mathbb{R}}

\begin{document}

\title[]
{Global existence for full compressible Navier-Stokes equations around the Couette flow with a temperature gradient in an infinite channel}
%{Global stability of the Couette flow for the full compressible Navier-Stoke equations in an infinite channel}
%{Global existence and low Mach number limit of full compressible Navier-Stokes equations around the Couette flow in an infinite channel}
\author{Tuowei Chen}
\address{Tuowei Chen, Department of Mathematics, Hong Kong University of Science and Technology, HongKong, P. R. China;}
\email{tuowei\_chen@163.com}
\author{Qiangchang Ju}
\address{Qiangchang Ju, Institute of Applied Physics and Computational Mathematics, 100088 Beijing, P. R. China;}
\email{ju\_qiangchang@iapcm.ac.cn}

\date{\today}

\begin{abstract}
This paper is concerned with the two-dimensional full compressible Navier-Stokes equations between two infinite parallel isothermal walls, where the upper wall is moving with a horizontal velocity, while the lower wall is stationary, and there allows a temperature difference between the two walls. It is shown that if the initial state is close to the Couette flow with a temperature gradient, then the global strong solutions exist, provided that the Reynolds and Mach numbers are low and the temperature difference between the two walls is small. The low Mach number limit of the global strong solutions is also shown for the case that both walls maintain the same temperature.
\end{abstract}

\maketitle

%\tableofcontents
%%%%%%%%%%%%%%%%%%%%%%%%%%%%%%%%%%%%%%%%%%%%%%%%%%%%%%%%%%%%%%%%%%%%%%%%%%%%%%%%%%%%%%%%%%%%%%%%%%
\section{Introduction}
This paper is concerned with the stability of the Couette flow with a temperature gradient in a two-dimensional infinite channel. This flow describes the motion of viscous, compressible and thermally conducting gas flows between two parallel isothermal walls, where the upper wall is moving with a constant horizontal velocity, while the lower wall is stationary, and there allows a temperature difference between the two walls. The gas flow is governed by the non-dimensional full compressible Navier-Stokes equations
\begin{equation}\label{eq:2DCNS}
	\begin{split}
		&\pt_t\rho+\div(\rho \bfu)=0,\\
		&\rho(\pt_t \bfu+ \bfu\nabla\bfu)+\ep^{-2}\nabla P(\rho,\fT)= \mu\Delta \bfu+(\mu+\mup)\nabla\div \bfu,\\
		&\frac{\rho}{\gamma-1}(\pt_t \fT+ \bfu\cdot\nabla \fT)+P(\rho,\fT)\div \bfu
		=\kappa\Delta\fT+\ep^2\big(2\mu|\Du|^2+\mup(\div \bfu)^2\big),
	\end{split}
\end{equation}
in a two-dimensional infinite channel
\begin{equation}\notag
	\Omega=\R\times (0,1).  
\end{equation}
Here, $\rho(\bfx,t)$, $\bfu(\bfx,t)=(u^1(\bfx,t),u^2(\bfx,t))^\top$ and $\fT(\bfx,t)$ are unknowns and represent the gas density, velocity and absolute temperature, respectively, at time $t\geq0$, and position $\bfx=(x_1,x_2)\in\Omega$. The gas pressure $P$ is given by the equation of state for ideal gas, 
\begin{equation}\notag
P=P(\rho,\fT)=\rho\fT.
\end{equation}
The deformation tensor $\Du$ is defined as
\begin{equation}\notag
	\Du:=\frac{1}{2}\big(\nabla \bfu+ (\nabla \bfu)^\top\big).
\end{equation}
Moreover, $\ep$, $\mu$, $\mup$ and $\kappa$ are non-dimensional parameters satisfying
\begin{equation}
	\ep=\sqrt{\gamma}\Ma, \quad \mu=\frac{1}{\Re}, \quad \mu+\mup\geq0,\quad \kappa=\frac{C_p}{\Re \Pr},  \label{ReMAnew}
\end{equation}
where $\gamma>1$ is the ratio of the specific heats, $\Ma>0$ and $\Re>0$ are the Mach and Reynolds numbers based on the reference state on the upper wall, $\Pr>0$ is the Prandtl number, and $C_p>0$ is the specific heat at constant pressure. 

The boundary conditions on $\pt\Omega=\{\bfx: x_2=0\}\cup\{\bfx: x_2=1\}$ are 
\begin{equation}\label{boundarycon}
	\begin{aligned}
		&u^1(x_1,1,t)=1, \quad  u^2(x_1,1,t)=0, \quad \fT(x_1,1,t)=1,\\
		 &u^1(x_1,0,t)=0, 
		 \quad u^2(x_1,0,t)=0,
		 \quad \fT(x_1,0,t)=\chi, \quad\quad\quad  \forall x_1\in \R,\, t>0,
	\end{aligned}
\end{equation}
where $\chi$ is a given positive constant. The initial condition is
\begin{equation}
(\rho,\bfu,\fT)^\top(\bfx,0)=(\rho_0,\bfu_0,\fT_0)^\top(\bfx), \quad\quad \forall \bfx\in \Omega. \label{initialcon}
\end{equation}
We refer the reader to \cite{CJ2025} for a detailed non-dimensionalization of the full compressible Navier-Stokes equations between two parallel isothermal walls with a relative velocity. As in \cite{CJ2025}, it is further assumed that the ratio $\frac{\mup}{\mu}$ is a constant. Particularly, if Stokes' hypothesis holds (see \cite{GA1999,Papalexandris2020}), we have $\frac{\mup}{\mu}=-\frac{2}{3}$, implying that the bulk viscosity coefficient, i.e., $\mup+\frac{2}{3}\mu$, is zero.

It is easy to observe that there exists an exact solution $\tbfW$ to the system \eqref{eq:2DCNS}--\eqref{initialcon} satisfying
\begin{equation}\label{def:Couette}
	\begin{aligned}
	&\tbfW=\tbfW(x_2)=(\trho,\tbfu,\tfT)^\top(x_2),\\		&\tfT=\chi +(1-\chi)x_2-\frac{\ep^2\Pr }{2C_p}(x^2_2-x_2),
	\quad \tu^1=x_2, \quad \tu^2=0,\quad\trho=(\tfT)^{-1}.
	\end{aligned}
\end{equation}
The above exact solution $\tbfW$ is known as the Couette flow with a temperature gradient \cite{MDA2008,DEH1994}. We point out that the Couette flow $\tbfW$ has been extensively employed as a benchmark for simulations of low Mach number viscous and thermally conducting fluid flows in the context of computational fluid dynamics, see \cite{Xu2001,KW2010,CYLLL2016}. We also mention that the Couette flow for the two-dimensional incompressible or isentropic compressible Navier-Stokes equations typically assumes a constant density and has a velocity profile $\bfu_E=(x_2,0)$, see \cite{CLDZ2020,Kagei2011JMFM,HLX2025}.  

In the literature, there have been many mathematical studies on the Couette flow for the incompressible Navier-Stokes equations. Romanov \cite{Romanov1973} first proved that the Couette flow is linearly stable for all Reynolds numbers. The transition threshold for the Couette flow in a finite channel was studied in \cite{CDZ2024,CLDZ2020}. More recent progress on the incompressible Couette flow can be found in \cite{MZ2020,WZZ2018,WZZ2019,WZ2021} and the references therein.

For the Couette flow governed by the isentropic compressible Navier-Stokes equations, Iooss and Padula \cite{IP1997} investigated the linear stability of stationary parallel flows in a cylindrical domain. Kagei \cite{Kagei2011JMFM} proved the asymptotic stability of the Couette flow in an infinite layer when the Reynolds and Mach numbers are small. Later, Kagei \cite{Kagei2011JDE,Kagei2012} extended this result to general parallel flows. Li and Zhang \cite{LZ2017} studied the stability of the Couette flow in an infinite layer with Navier-slip boundary conditions at the bottom boundary. Readers are referred to \cite{KN2015,KN2019,KT2020} for further results on the Poiseuille flow and the Couette flow between two concentric cylinders.  Recently, the linear stability for the Couette flow in isentropic compressible fluids at high Reynolds number was investigated in \cite{ADM2021,ZZZ2022}. Huang, Li and Xu \cite{HLX2025} studied the nonlinear stability of the two-dimensional compressible Couette flow at high Reynolds numbers in an infinitely long flat torus $\mathbb{T}\times\R$.

However, there is significantly less literature concerning the mathematical study of the Couette flow for the full compressible Naiver-Stokes equations. The Couette flow $\tbfW$ includes variations in density and temperature, being more complex than the one governed by the incompressible or isentropic compressible Navier-Stokes equations, which has a constant density. For the Couette flow $\tbfW$ with small initial perturbations, the present authors \cite{CJ2025} proved the global existence of strong solutions to the equation \eqref{eq:2DCNS} between two finite parallel walls with horizontal periodic boundary conditions, provided that the Reynolds and Mach numbers are low. In the same paper \cite{CJ2025}, the present authors also studied the low Mach number limit of the resulting global strong solutions. We refer the reader to \cite{Bessaih1995,FJX2024,Ou2009,JO2011,JO2022,MRS2022,OY2022} for studies on the low Mach number limit of the compressible Navier-Stokes equations. It is particularly interesting and more difficult to study the global stability of the Couette flow between two infinite parallel walls with inhomogeneous Dirichlet boundary conditions, which is the main concern of the present paper.

In this paper, we aim to prove the global existence of strong solutions to the system \eqref{eq:2DCNS}--\eqref{initialcon} around the Couette flow with a temperature gradient, i.e., $\tbfW$ in \eqref{def:Couette}. We also study the low Mach number limit of the resulting global strong solutions.

Before stating the main theorems, we introduce the notations used in this paper. \\
$\bullet\,\,\,L^q(\Omega)$: The standard Lebesgue space over $\Omega$ with the norm $\| \cdot \|_{L^{q}}\,\,(1\leq q\leq \infty)$.\\
$\bullet\,\,\,H^l(\Omega)$: The usual $L^2$-Sobolev space over $\Omega$ of integer order $l$ with the norm $\| \cdot \|_{H^{l}}\,\,(l\geq 0)$.\\
$\bullet\,\,\,C([0,T];H^l(\Omega))$: The space of continuous functions on an interval $[0,T]$ with values in $H^l(\Omega)$.\\
The function spaces $L^2([0,T];H^1(\Omega))$ and $L^\infty([0,T];H^2(\Omega))$ can be defined similarly.

We denote by $\pt_{1}$ and $\pt_{2}$ the operators $\pt_{x_1}$ and $\pt_{x_2}$, respectively. For a function $f(\bfx)$ on $\Omega$, we adopt the simplified notation
\begin{equation}\notag
	\int f \md \bfx:=\int_{\Omega} f \md \bfx.
\end{equation} Moreover, the perturbation functions are defined as
\begin{equation}\label{def-0}
	(\varphi(t),\bfpsi(t),\theta(t))^\top:=(\rho(t)-\trho,\bfu(t)-\tbfu,\fT(t)-\tfT)^\top,
	\quad \forall t\geq0.  
\end{equation}

The main theorem of this paper is stated as follows.
\begin{theorem} \label{main1}	
	Suppose that $\chi$, the non-dimensional temperature on the lower wall, satisfies
	\begin{equation} \label{chi}
		|1-\chi|=O(\ep). 
	\end{equation}
	Suppose that the initial perturbation $(\varphi_0,\bfpsi_0,\theta_0):=(\varphi(0),\bfpsi(0),\theta(0))$ satisfies
	\begin{equation} \label{compatibilitycon}
			\varphi_0\in H^3(\Omega),\quad \bfpsi_0\in  H^1_0(\Omega)\cap H^3(\Omega),
			\quad \theta_0\in  H^1_0(\Omega)\cap H^3(\Omega).
	\end{equation}
	We further assume that the initial perturbation satisfies the compatibility condition
	\begin{equation}\label{compatibilitycon1}
		\pt_{t}\bfpsi(0)=0,\quad  
		\pt_{t}\theta(0)=0, \quad\quad\text{on}\quad\pt\Omega.	 
	\end{equation}
	Then, there exist positive constants $\Re^\prime$, $\vep^\prime$ and $C_0$ such that if
	\begin{equation}
		\Re<\Re^\prime, \quad \ep<\vep^\prime,  \label{con2}
	\end{equation}
	and if
	\begin{equation}
		\begin{aligned}
			&\|\bfpsi(0)\|^2_{L^2}+\ep^{2}\|\nabla\bfpsi(0)\|^2_{L^2}
			+\ep^{4}\|(\nabla^2\bfpsi,\pt_{t}\bfpsi)(0)\|^2_{L^2}
			+\ep^{6}\|(\nabla^3\bfpsi,\nabla\pt_{t}\bfpsi)(0)\|^2_{L^2}\leq C_0\ep^2,\\
			&\|\varphi(0)\|^2_{L^2}+\ep^{2}\|\nabla\varphi(0)\|^2_{L^2}
			+\ep^{4}\|(\nabla^2\varphi,\pt_{t}\varphi)(0)\|^2_{L^2}
			+\ep^{6}\|\nabla^3\varphi(0)\|^2_{L^2}\leq C_0\ep^4,\\
			&\|\theta(0)\|^2_{L^2}+\ep^{2}\|\nabla\theta(0)\|^2_{L^2}
			+\ep^{4}\|(\nabla^2\theta,\pt_{t}\theta)(0)\|^2_{L^2}
			+\ep^{6}\|(\nabla^3\theta,\nabla\pt_{t}\theta)(0)\|^2_{L^2}\leq C_0\ep^4,
		\end{aligned}  \label{con3}
	\end{equation}
	then there exists a unique global strong solution $(\rho,\bfu,\fT)$ to \eqref{eq:2DCNS}--\eqref{initialcon} satisfying 
	\begin{equation}
		\begin{aligned}
			&\varphi\in C\big([0,\infty);H^3(\Omega)\big),
			\quad \bfpsi\in C\big([0,\infty);H^3(\Omega)\big)\cap L^2\big([0,\infty);H^4(\Omega)\big),\\
			&\theta\in C\big([0,\infty);H^3(\Omega)\big)\cap L^2\big([0,\infty);H^4(\Omega)\big),
		\end{aligned}   \label{regularity}
	\end{equation}
	and
	\begin{equation}
		\begin{aligned}
			&\sup_{t\in [0,\infty)}\big(\ep^{-1}\|\varphi(t)\|_{L^2}+\|\bfpsi(t)\|_{L^2}
			+\ep^{-1}\|\theta(t)\|_{L^2}\big)\leq \widehat{C}\ep,\\
			&\sup_{t\in [0,\infty)}\big(\ep^{-1}\|\nabla\varphi(t)\|_{L^2}
			+\|\nabla\bfpsi(t)\|_{L^2}+\ep^{-1}\|\nabla\theta(t)\|_{L^2}\big)	\leq \widehat{C}.
		\end{aligned}  \label{uniform}
	\end{equation}
	Here, $\widehat{C}$ is a positive constant independent of $\ep$. In addition, we have
	\begin{equation}
	\|\big(\bfu(t)-\tbfu,\fT(t)-\tfT\big)\|_{L^\infty}\rightarrow 0 
		\quad\quad \text{as} \quad t\rightarrow \infty. \label{largetime}
	\end{equation}
\end{theorem}

\begin{remark}
	It follows from \eqref{def:Couette} and \eqref{def-0} that $\pt_{t}\bfpsi(0)=\pt_{t}\bfu(0)$. The notation $\pt_{t}\bfu(0)$ is defined by taking $t = 0$ in $\eqref{eq}_2$, that is,
	\begin{equation}
		\rho_0(\pt_t \bfu(0)+ \bfu_0\cdot\nabla \bfu_0)+\ep^{-2}\nabla P(\rho_0,\fT_0)= \mu\Delta \bfu_0+(\mu+\mup)\nabla\div \bfu_0.
	\end{equation}
	It is indeed a compatibility condition. The notations $\pt_{t}\varphi(0)$ and $\pt_{t}\theta(0)$ are defined in a similar way. 
\end{remark}

\begin{remark}
	From the proof below, the constant $\ep^\prime$ indeed depends on the Reynolds number $\Re$ in \Cref{main1}. Moreover, the constant $C_0$ depends on $\Re$, but is independent of $\ep$.  
\end{remark}

In the absence of a temperature difference between the upper and lower walls, i.e., it holds that
\begin{equation}
	\chi=1,  \label{same}
\end{equation}
we state the result on the low Mach number limit of the global strong solutions obtained in \Cref{main1} as follows. 
\begin{theorem} \label{main2}
	Suppose that \eqref{same} holds, and suppose that the conditions \eqref{compatibilitycon}--\eqref{con3} hold as in \Cref{main1}. Denote by $(\rho^\ep,\bfu^\ep,\fT^\ep)$ the unique global strong solution obtained from \Cref{main1}. Then, it holds that
	\begin{equation}
		\|\rho^\ep-1\|_{L^2}=O(\ep^2),\quad	\|\bfu^\ep-\tbfu\|_{L^2}=O(\ep),\quad	\|\fT^\ep-1\|_{L^2}=O(\ep^2). \label{converge0}
	\end{equation}
	Note that $(\tbfu,p_*)$ with $p_*:= P(1,1)$ is indeed the incompressible Couette flow being a exact solution to the steady incompressible Navier-Stokes system,
	\begin{equation}
		\left.\begin{cases}
			\div \bfu=0,    \quad\quad\text{in}\quad \Omega,  \\
			\bfu\nabla \bfu+\nabla p = \mu\Delta \bfu,\quad\quad\text{in}\quad \Omega,\\
			\bfu|_{x_2=1}=(1,0)^\top,\\
			\bfu|_{x_2=0}=0.
		\end{cases}\label{2DINS}
		\right.
	\end{equation} 
	Thus, the global strong solution $(\rho^\ep,\bfu^\ep,\fT^\ep)$ strongly converges to $(1,\tbfu,1)$, i.e., the incompressible Couette flow, in the $L^2$ space as $\ep\rightarrow0$.
\end{theorem}

Now we make some comments on the analysis of this paper. Different from the problems concerning with small initial perturbations around a constant state \cite{MN1980,MN1983} or the Couette flow of isentropic compressible fluids \cite{Kagei2011JMFM}, the present work considers a background solution with variations in both density and temperature, see \eqref{def:Couette}. The temperature gradient of the Couette flow $\tbfW$ yields a perturbation system with a variable coefficient term of $\varphi$, i.e., $\nabla\tfT\varphi$, which is derived from the gradient term of the pressure, see \eqref{eq} below. Due to the presence of the term $\nabla\tfT\varphi$, the Matsumura-Nishida type energy method used in \cite{MN1980,MN1983,Kagei2011JMFM} can not be directly employed to obtain the time-decay estimate for $\nabla \varphi$. In our previous work \cite{CJ2025} on the Couette flow in a finite channel, such difficulty has been partially overcome by introducing a zero mass perturbation condition and using the Poincar\'{e} inequality to obtain a $L^2H^3$-type estimate for $\varphi$. However, for the present case in an infinite channel, it can not be expected that the Poincar\'{e} inequality holds for $\varphi$ in the unbounded domain $\Omega$. Thus, the approach used in \cite{CJ2025} can not be applied here. To address this issue, we establish a $L^2L^2$-type estimate for the term $\tfT\pt_2\varphi+\varphi\pt_2\tfT$ by using a weighted energy method. Then, we show the $L^2L^2$-type estimate for $\pt_{1}\varphi$ by studying a steady inhomogeneous Stokes system between two infinite parallel walls. These time-decay estimates for the terms $\tfT\pt_2\varphi+\varphi\pt_2\tfT$ and $\pt_{1}\varphi$ play a key role in establishing the higher order uniform \textit{a priori} estimates. Moreover, to overcome the drift effect terms from the non-zero background velocity $\tbfu$, we employ the relative entropy method \cite{Dafermos1979} to establish the basic energy estimate, provided that the Reynolds number is sufficiently low. In addition, to study the low Mach number limit of the strong solutions, we construct a $\ep$-weighted energy functional to derive uniform bounds independent of $\ep$.

\begin{remark}
	The approach developed in the present paper can be applied to the case for the Couette flow $\tbfW$ in a finite channel, removing the zero mass perturbation condition in the previous work \cite{CJ2025}, which requires the initial mass being equal to the mass of the Couette flow $\tbfW$. In addition, the results in \Cref{main1,main2} can be extended to the case in three space dimensions by slight modifications. 
\end{remark}

The rest of this paper is organized as follows. In Section 2, we collect some elementary inequalities and introduce the system for perturbations. In Section 3, we show the uniform \textit{a priori} estimates by using a $\ep$-weighted energy method. Finally, the main theorems are proved in Section 4.

\section{Preliminaries}
In this section we first recall some known facts and elementary inequalities that will be used later. Later in the section, we give the system of equations for perturbations and show the local existence and uniqueness of strong solutions to the initial-boundary value problem of the resulting system.

\subsection{Some elementary inequalities}
We first recall the following Poincar\'{e} type inequality on an infinite channel $\Omega$.  
\begin{lemma} \label{Poincareineq2}
	For $f\in W^{1,p}_0(\Omega)$, $1\leq p\leq\infty$, it holds that
	\begin{equation}
		\|f\|_{L^q(\Omega)}\leq \frac{h}{2}\|\nabla f\|_{L^q(\Omega)}.
	\end{equation}
\end{lemma}

\begin{proof}
	See Theorem II.5.1 in \cite{Galdi2011}.
\end{proof}
The following is the Gagliardo-Nirenberg inequality which will be used frequently.
\begin{lemma}\label{GN lemma}
	For $p\in[2,\infty)$, there exists a generic positive constant $C$ which may depend on $p$ and $\Omega$, such that for any $f\in H^2(\Omega)$, we have
	\begin{equation}
		\begin{aligned}
			&\|f\|_{L^p}    
			\leq C\|f\|^{\frac{2}{p}}_{L^2}\|\nabla f\|^{1-\frac{2}{p}}_{L^2}+C\|f\|_{L^2},\\
			&\|f\|_{L^\infty}\leq C\|f\|^{\frac{1}{2}}_{L^2}
			\|\nabla^2 f\|^{\frac{1}{2}}_{L^{2}}+C\|f\|_{L^2}.
		\end{aligned}
	\end{equation}
\end{lemma}
\begin{proof}
	See \cite{Nirenberg1959}.
\end{proof}

Next, we consider the steady inhomogeneous Stokes system on $\Omega$,
\begin{equation}\label{eq:Stokes}
	\left.
	\begin{cases}
		\div \bfu=g\quad\quad\mathrm{in} \,\,\Omega,\\
		-\mu\Delta \bfu+\nabla p=\bfF \quad\quad\mathrm{in} \,\,\Omega,\\
		\bfu=0 \quad\quad\mathrm{on}\,\,\pt\Omega,\\
		\bfu\rightarrow 0\quad\quad \mathrm{as} \,\,|\bfx|\rightarrow \infty,
	\end{cases}
	\right.
\end{equation}
where $g$ and $\bfF$ are known functions. We have the following lemma.  
\begin{lemma}\label{Stokes lemma}
	Assume that $\bfF\in L^2(\Omega)$ and $g\in H^{1}(\Omega)$. Let $(\bfu,p)$ be a smooth solution to the boundary value problem \eqref{eq:Stokes} and suppose that $\bfu\in H^2(\Omega)$ and $\nabla p\in L^2(\Omega)$. Then it holds that
	\begin{equation} \label{ineq:Stokes}
		\|\nabla^2\bfu\|_{H^2}+\|\nabla p\|_{L^2}
		\leq C\big(\| \bfF\|_{L^2}+\| g\|_{H^1}\big),
	\end{equation}
	where $C$ is a positive constant depending only on $\Omega$ and $\mu$.
\end{lemma}
\begin{proof}
	Let $\Omega_i$, for $i=1,2,3,\dots$, be a sequence of bounded open subsets of $\Omega$, each having a smooth boundary and satisfying
	\begin{equation}\notag
		\Omega_i\subset\Omega_{i+1},\quad\bigcup_{i=1}^{\infty}\Omega_i=\Omega,\quad
		(\pt\Omega_i\setminus\pt\Omega)\cap B_i=\emptyset,
	\end{equation}
	where $B_{i}=\{\bfx\in\Omega:|\bfx|>i\}$. For example, $\Omega_{i}$ can be obtained by making slight modifications to the open set $\Omega\cap B_{i+1}$ to ensure that the resulting set has a smooth boundary.
	
	Then, for $i=1,2,3,\dots$, let
	\begin{equation}\notag
		\bfu_{i}(\bfx):=\bfu|_{\Omega_{i}}(\bfx), \quad p_{i}(\bfx):=p|_{\Omega_{i}}(\bfx),\quad\quad \forall \bfx\in\Omega_{i}.
	\end{equation}
	It is observed that $(\bfu_{i},p_{i})$ is a solution to the boundary value problem
	\begin{equation}\label{eq:Stokes-i}
		\left.
		\begin{cases}
			\div \bfu_{i}=g|_{\Omega_{i}}
			\quad\quad\mathrm{in} \,\,\Omega_{i},\\
			-\mu\Delta \bfu_{i}+\nabla p_{i}=\bfF|_{\Omega_{i}} \quad\quad\mathrm{in} \,\,\Omega_{i},\\
			\bfu_{i}=0 \quad\quad\mathrm{on}\,\,\pt\Omega_{i}\cap\pt\Omega,\\
			\bfu_{i}=\bfu
			\quad\quad\mathrm{on}\,\,\pt\Omega_{i}\setminus\pt\Omega,
		\end{cases}
		\right.
	\end{equation}
	where
	\begin{equation}\notag
		g_{i}:=g|_{\Omega_{i}},\quad 
		\bfF_{i}:=\bfF|_{\Omega_{i}}.
	\end{equation}
	By the elliptic theory for the inhomogeneous Stokes system in a bounded domain (see the Chapters four in \cite{Galdi2011}), it holds that
	\begin{equation}\label{ineq:Stokes-i}
	\|\nabla^2\bfu_{i}\|_{L^2(\Omega_{i})}
	+\|\nabla p_{i}\|_{L^2(\Omega_{i})}
	\leq C\big(\| \bfF_{i}\|_{L^2(\Omega_{i})}
	+\| g_{i}\|_{H^1(\Omega_{i})}
	+\|\bfu\|_{H^{\frac{3}{2}}(\pt\Omega_{i}\setminus\pt\Omega)}\big),
\end{equation}
	where $C$ is a positive constant independent of $i$.
	
	It follows from the trace theorem (c.f.\cite{Evans2010}) that
	\begin{equation}\notag
		\|\bfu\|_{H^{\frac{3}{2}}(\pt\Omega_{i}\setminus\pt\Omega)}
		\leq C\|\bfu\|_{H^2(\Omega\setminus\Omega_{i})},
	\end{equation}
	where $C$ is a positive constant independent of $i$. This, together with the fact that
	\begin{equation}\notag
		\lim\limits_{i\rightarrow\infty}\|\bfu\|_{H^2(\Omega\setminus\Omega_{i})}=0,
	\end{equation}
	gives
	\begin{equation}\notag
		\lim\limits_{i\rightarrow\infty}
		\|\bfu\|_{H^{\frac{3}{2}}(\pt\Omega_{i}\setminus\pt\Omega)}=0.
	\end{equation}
	
	Based on the above equation, letting $i\rightarrow\infty$ on both sides of the inequality \eqref{ineq:Stokes-i} implies \eqref{ineq:Stokes}. This finishes the proof.
\end{proof}

Furthermore,  we recall the classical elliptic theory for the Lam\'{e} system.
\begin{lemma}\label{Lame lemma}
	Let $\bfu$ be a smooth solution solving the problem
	\begin{equation}
		\left.
		\begin{cases}
			\mu\Delta \bfu+(\mu+\mup)\nabla\div 
			\bfu=\bfF\quad\quad\mathrm{in}\,\,\Omega,\\
				\bfu=0 \quad\quad\mathrm{on}\,\,\pt\Omega,\\
			\bfu\rightarrow 0\quad\quad \mathrm{as} \,\,|\bfx|\rightarrow \infty.
		\end{cases} \label{lamesys}
		\right.
	\end{equation} 
	Then, for $p\in[2,\infty)$ and a integer $k\geq 0$, there exists a positive constant $C$ depending only on $\mu$, $\mup$, $p$, $k$ and $\Omega$ such that the following estimates hold:
	
	\noindent(1) if $\bfF\in L^p(\Omega)$, then
	\begin{equation}
		\|\bfu\|_{W^{k+2,p}}\leq C\|\bfF\|_{W^{k,p}};
	\end{equation}
	\noindent(2) if $\bfF=\div \bff$ with $\bff=(\bff^{ij})_{3\times 3}$, $\bff^{ij}\in W^{k,p}(\Omega)$, then
	\begin{equation}
		\|\bfu\|_{W^{k+1,p}}\leq C\|\bfF\|_{W^{k,p}}.
	\end{equation}
\end{lemma}
\begin{proof}
	See the standard $L^p$-estimates for the Agmon-Douglis-Nirenberg systems \cite{ADN1959,ADN1964}.
\end{proof}

\subsection{System of equations for the perturbation}
Recalling the notation \eqref{def-0}, we rewrite the system \eqref{eq:2DCNS} as
\begin{equation} \label{eq}
		\begin{split}
		&\pt_t\varphi+\bfu\cdot\nabla\varphi+\rho\div\bfpsi+\bfpsi\cdot\nabla\trho=0,  \\
		&\rho\big(\pt_t\bfpsi+\bfu\nabla \bfpsi+\bfpsi\nabla\tbfu\big)
		+\ep^{-2}\nabla P(\rho,\fT)-\mu\Delta\bfpsi-(\mu+\mup)\nabla\div\bfpsi=0,\\
		&\frac{1}{\gamma-1}\rho\big(\pt_t\theta+ \bfu\cdot\nabla\theta+\bfpsi\cdot\nabla\tfT\big)
		+ P(\rho,\fT)\div\bfpsi-\kappa\Delta\theta\\
		&= \ep^2\big(2\mu\Dpsi:\Dpsi+4\mu\Dtbfu:\Dpsi+\mup(\div\bfpsi)^2\big).
	\end{split}
\end{equation}

We consider the system \eqref{eq} under boundary conditions
\begin{equation}
	\bfpsi=0,\quad  
\theta=0, \quad\quad\text{on}\quad\pt\Omega,	  \label{eq4}
\end{equation}
and the initial condition
\begin{equation}
	(\varphi,\bfpsi,\theta)^\top(0)=(\varphi_0,\bfpsi_0,\theta_0)^\top.  \label{eq5}
\end{equation}

We state the local existence result as follows.
\begin{proposition}[Local existence and uniqueness]\label{local}
	Suppose that \eqref{compatibilitycon}--\eqref{compatibilitycon1} hold. Then, there exists a small time $T$ such that there exists a local unique strong solution $(\varphi,\bfpsi,\theta)$ to \eqref{eq}--\eqref{eq5} on $\Omega\times[0,T]$ satisfying
	\begin{equation}
		\varphi\in C([0,T];H^3(\Omega)),\,\,\,\,\,
		\bfpsi,\theta\in C([0,T];H^3(\Omega))\cap L^2(0,T;H^4(\Omega)).\label{local0}
	\end{equation}
\end{proposition}

\begin{proof}
	The proof of this proposition can be done by using the linearization technique, classical energy method and Banach fixed point argument as in \cite{MN1980,MN1983}. For simplicity, the details are omitted here.  
\end{proof}

\section{Uniform estimates}
In this section, we derive the uniform estimates (in time) for smooth solutions to the initial-boundary value problem \eqref{eq}--\eqref{eq5}.

In the following of this section, we fix a smooth solution $(\varphi,\bfpsi,\theta)$ to \eqref{eq}--\eqref{eq5} on $\Omega\times[0,T]$ for a time $T>0$, and assume that the conditions \eqref{chi}--\eqref{compatibilitycon1} hold. For $t\in[0,T]$, we define 
\begin{equation}\label{def:A0123}
	\begin{aligned}
		A_0(t):=&\sup_{s\in [0,t]}\big(\|\bfpsi(s)\|^2_{L^2}+\ep^{-2}\|\varphi(s)\|^2_{L^2}+\ep^{-2}\|\theta(s)\|^2_{L^2}\big)\\
		&+\int_{0}^{t}\big(\|\bfpsi(s)\|^2_{H^1}+\ep^{-2}\|\theta(s)\|^2_{H^1}\big)\md s,\\
		A_1(t):=&\sup_{s\in [0,t]}\big(\|\nabla\bfpsi(s)\|^2_{L^2}+\ep^{-2}\|\nabla\theta(s)\|^2_{L^2}  \big)
		+\int_{0}^{t}\big(\|\pt_t\bfpsi\|^2_{L^2}+\ep^{-2}\|\pt_t\theta\|^2_{L^2}\big)\md s,\\
		A_2(t):=&\ep^{-2}\sup_{s\in [0,t]}\|\nabla\varphi(s)\|^2_{L^2}
		+\int_{0}^{t}\big(\ep^{-4}\|\pt_1\varphi\|^2_{L^2}
		+\ep^{-4}\|\tfT\pt_2\varphi+\varphi\pt_2\tfT\|^2_{L^2}\big)\md s\\
		&+\int_{0}^{t}\big(\ep^{-2}\|\pt_{t}\varphi\|^2_{L^2}
		+\|\bfpsi(s)\|^2_{H^2}+\ep^{-2}\|\theta(s)\|^2_{H^2}\big)\md s,\\
		A_3(t):=&\sup_{s\in[0,t]}\big(\ep^{-2}\|\pt_{t}\varphi(s)\|^2_{L^2}
		+\|\pt_{t}\bfpsi(s)\|^2_{L^2}+\ep^{-2}\|\pt_{t}\theta(s)\|^2_{L^2}  \big)\\
		&+\int_{0}^{t}\big(\|\nabla\pt_{t}\bfpsi(s)\|^2_{L^2}
		+\ep^{-2}\|\nabla\pt_{t}\theta(s)\|^2_{L^2}\big)\md s,
	\end{aligned}
\end{equation}
and
\begin{equation}\label{def:A45}
	\begin{aligned}
		A_4(t):=&\sup_{s\in[0,t]}\big(\ep^{-2}\|\nabla^2\varphi(s)\|^2_{L^2}
		+\|\nabla^2\bfpsi(s)\|^2_{L^2}+\ep^{-2}\|\nabla^2\theta(s)\|^2_{L^2}  \big)\\
		&+\int_{0}^t\big(\ep^{-4}\|\nabla\pt_{1}\varphi(s)\|^2_{L^2}
		+\ep^{-4}\|\pt_{2}(\tfT\pt_2\varphi+\varphi\pt_2\tfT)\|^2_{L^2}\big)\md s\\
		&+\int_{0}^t\big(\|\bfpsi(s)\|^2_{H^3}
		+\ep^{-2}\|\theta(s)\|^2_{H^3}\big)\md s,\\
		A_5(t):=&\sup_{s\in[0,t]}\big(\ep^{-2}\|\nabla^3\varphi(s)\|^2_{L^2}
		+\|(\nabla^3\bfpsi,\nabla\pt_{t}\bfpsi)(s)\|^2_{L^2}+\ep^{-2}\|(\nabla^3\theta,\nabla\pt_{t}\theta)(s)\|^2_{L^2} \big)\\
		&+\int_{0}^{t}\big( \|\pt^2_{t}\bfpsi(s)\|^2_{L^2}+\ep^{-2}\|\pt^2_{t}\theta(s)\|^2_{L^2}
		+\|\bfpsi(s)\|^2_{H^4}
		+\ep^{-2}\|\theta(s)\|^2_{H^4}\big)\md s.
	\end{aligned}
\end{equation}
We also define
\begin{equation}\label{def:N}
	N(t):=\ep^{-2}A_0(t)+A_1(t)+A_2(t)+\ep^{2}A_3(t)+\ep^{2}A_4(t)+\ep^{4}A_5(t), 
	\quad\forall t\in[0,T].
\end{equation}

For simplicity, we use the notation 
\begin{equation}\label{Nhat}
	\Nhat:=N(T),  
\end{equation}
and denote by $C$ a generic positive constant depending only on $\Re$, $\Pr$, $\frac{\mup}{\mu}$ and $\gamma$, but not on $\ep$ and $T$. In addition, we denote by $\tC$ a generic positive constant depending only on $\Pr$, $\frac{\mup}{\mu}$ and $\gamma$, but not on $\Re$, $\ep$ and $T$.

Moreover, recalling \eqref{chi}, since we focus on the case with a low Mach number and small perturbations, we always assume that
\begin{equation}\label{focus}
	\ep\leq 1,\quad |1-\chi|\leq \frac{1}{2},  \quad \Nhat\leq 1.  
\end{equation} 

The main result of this section can be concluded as follows.
\begin{proposition} \label{priori}
	Suppose that \eqref{chi}--\eqref{compatibilitycon1} holds. Then, there exists a positive constant $\Re^\prime$ depending only on $\Pr$, $\frac{\mup}{\mu}$ and $\gamma$, and positive constants $\vep^\prime$ and $N^\prime$ depending only on $\Re$, $\Pr$, $\frac{\mup}{\mu}$ and $\gamma$, such that if 
	\begin{equation}
		\Re<\Re^\prime, \quad \ep<\vep^\prime, \quad \Nhat <N^\prime,
	\end{equation}
	then it holds that
	\begin{equation}
		\Nhat\leq \hat{C}N(0).  
	\end{equation}
	Here, $\hat{C}$ is a positive constant depending only on $\Re$, $\Pr$, $\frac{\mup}{\mu}$ and $\gamma$.
\end{proposition}
\begin{proof}
	\Cref{priori} follows from \Cref{basic}, \Cref{lemmaH1}, \Cref{lemmaH2,lemmaH3} below.                           
\end{proof}

\subsection{Basic energy estimate}
This subsection is devoted to establishing the basic energy estimate for $A_0(t)$. 

We start with the following Poincar\'{e} type inequality. 
\begin{lemma}  \label{Poin}
	Suppose that the compatibility conditions \eqref{compatibilitycon}--\eqref{compatibilitycon1} hold. Then, the following Poincar\'{e} type inequalities hold:  
	\begin{equation}
		\|\bfpsi(t)\|_{L^2}\leq \tC \| \nabla\bfpsi(t)\|_{L^2},\,\,\,\,\, \|\theta(t)\|_{L^2}\leq \tC \| \nabla\theta(t)\|_{L^2},
	\quad\quad \forall \, t \in [0, T], \label{Poincare}
	\end{equation}	
\end{lemma}
\begin{proof}
	The inequality \eqref{Poincare} follows directly from the boundary condition \eqref{eq4} and \Cref{Poincareineq2}. 
	
	The proof is completed.
\end{proof}

Next, we show some elementary observations which will be used frequently.
\begin{lemma}\label{elementary}
	Suppose that \eqref{chi} holds. Then, there exists a positive constant $\vep_0$ depending only on $\Pr$, $\gamma$ and $\frac{\mup}{\mu}$, such that if $\ep<\vep_0$, then	
	\begin{equation}   \label{ob-0}
			\begin{aligned}
		&\inf_{x\in\Omega}\trho(x)>\frac{3}{4},	\quad\inf_{x\in\Omega}\tfT(x)>\frac{3}{4},
		\quad \|\trho-1\|_{L^\infty}\leq \tC\ep,
		\quad \|\tfT-1\|_{L^\infty}\leq \tC\ep,  \\ 
		&\|\nabla^k\trho\|_{L^\infty}\leq \tC\ep^k, 
		\quad \|\nabla^k\tfT\|_{L^\infty}\leq \tC\ep^k,  \quad\quad\quad\quad\quad\quad k=1,2,3,\\ 
	\end{aligned}
	\end{equation}
	and
	\begin{equation} \label{ob-1}
		\begin{aligned}
			&\inf_{(x,t)\in\Omega\times[0,T]}\rho(x,t)>\frac{1}{2},	\quad\inf_{(x,t)\in\Omega\times[0,T]}\fT(x,t)>\frac{1}{2},
			\quad \|(\varphi,\theta)(t)\|^2_{L^\infty}\leq \tC\Nhat\ep^2,\\
			& \|\bfpsi(t)\|^2_{L^\infty}\leq \tC\Nhat, \quad
			\|(\nabla\varphi,\nabla\theta)(t)\|^2_{L^\infty}\leq \tC\Nhat,
			\quad \|\nabla\bfpsi(t)\|^2_{L^\infty}\leq \tC\Nhat\ep^{-2},\quad\forall t\in[0,T]. 
		\end{aligned}
	\end{equation}
\end{lemma}
\begin{proof}
	Recalling of \eqref{def:Couette}, we first obtain from \eqref{chi} that
	\begin{equation}
		\|\tfT-1\|_{L^\infty}\leq |\chi-1|+\tC\ep^2\leq \tC\ep+\tC\ep^2\leq \tC\ep.  \notag
	\end{equation}
	This yields
	\begin{equation}
		\inf_{x\in\Omega}\tfT(x)>\frac{3}{4}, \label{ob1}
	\end{equation}
	provided that $\ep$ is small enough. Based on an analogous argument, we derive \eqref{ob-0}.
	
	Next, it follows from \eqref{def:N} that
	\begin{equation}
		\|\theta(t)\|^2_{L^2}\leq \Nhat\ep^{4},\quad 	\|\nabla^2\theta(t)\|^2_{L^2}\leq \Nhat,\quad\forall t\in[0,T],  \notag
	\end{equation}
	which, together with \Cref{GN lemma}, leads to
	\begin{equation}
		\sup_{t\in [0,T]}\|\theta(t)\|^2_{L^\infty}
		\leq \tC\sup_{t\in [0,T]}(\|\theta\|_{L^2}\|\nabla^2\theta\|_{L^2}
		+\|\theta\|^2_{L^2})\leq  \tC \Nhat\ep^2 \leq \frac{1}{4},  \label{ob2}
	\end{equation}
	provided that $\ep$ is small enough. Thus, \eqref{ob1} and \eqref{ob2} imply
	\begin{equation}
		\inf_{(x,t)\in\Omega\times[0,T]}\fT(x,t)\geq \inf_{x\in\Omega}\tfT(x)-\sup_{t\in [0,T]}\|\theta(t)\|_{L^\infty}>\frac{1}{2}.
	\end{equation}
	Other inequalities in \eqref{ob-1} can be obtained similarly.
	
	The proof is completed.
\end{proof}

We are in a position to establish the basic energy estimate by using the relative entropy method which was used in \cite{Dafermos1979,CJ2025}.
\begin{lemma}   \label{basic}
	Suppose that \eqref{chi} holds. Then, there exist positive constants $\Re_0$ and $\vep_1$ depending only on $\Pr$, $\gamma$ and $\frac{\mup}{\mu}$, such that if $\Re<\Re_0$ and $\ep<\vep_1$, then	
	\begin{equation}\label{eneA0}
		A_0(t)\leq CA_0(0), \quad\quad \forall t\in[0,T], 
	\end{equation}
	where the definition of $A_0$ is given in \eqref{def:A0123}.
\end{lemma}
\begin{proof}
	We introduce the relative entropy defined by 
	\begin{equation}
		\begin{aligned}
			\eta:=&\frac{\ep^2\rho}{\tfT}|\bfpsi|^2
			+\frac{\rho}{\gamma-1}\big(\frac{\fT}{\tfT}-\ln(\frac{\fT}{\tfT})-1\big)
			+\rho\big(\frac{\tau}{\ttau}-\ln(\frac{\tau}{\ttau})-1\big)   \\
			=&\frac{\ep^2\rho}{\tfT}|\bfpsi|^2
			+\frac{\rho}{\gamma-1}\big(\frac{\theta}{\tfT}-\ln(1+\frac{\theta}{\tfT})\big)
			+\rho\big(-\frac{\varphi}{\rho}-\ln(1-\frac{\varphi}{\rho})\big),  \label{equiv1}
		\end{aligned}
	\end{equation}
	where
	\begin{equation}
		\tau:=\frac{1}{\rho},\quad \ttau:=\frac{1}{\trho}. \notag
	\end{equation}
	For the function
	\begin{equation}
		f(z)=z-\ln(1+z),\quad\quad z\in(-1,\infty), \notag
	\end{equation}
	it holds that 
	\begin{equation}
		\begin{aligned}
			&f(0)=0,\quad\quad f^\prime(0)=0,\\
			&f^{\prime\prime}(z)=\frac{1}{(1+z)^2}>0,\quad\quad \forall z\in(-1,\infty), \notag
		\end{aligned}
	\end{equation}
	which yields that there exists a small constant $\delta_0>0$ such that 
	\begin{equation}
		f(z)\geq z^2, \quad\quad\forall z\in(-\delta_0,\,\delta_0). \notag
	\end{equation}
	Therefore, we deduce from \eqref{ob-1} and \eqref{equiv1} that
	\begin{equation}
		\tC^{-1}\|(\varphi,\ep\bfpsi,\theta)(t)\|^2_{L^2}
		\leq \int \eta(t)\md\bfx
		\leq \tC\|(\varphi,\ep\bfpsi,\theta)(t)\|^2_{L^2},\quad \forall t\in[0,T], \label{equiv2}
	\end{equation}
	provided that $\ep$ is small enough.
	
	Differentiating $\eta$ with respect to $t$ gives
	\begin{equation}
		\pt_t\eta=\left(\ln(\frac{\rho}{\trho})+\frac{\ep^2}{2}\frac{|\bfpsi|^2}{\tfT}
		+\frac{1}{\gamma-1}(\frac{\fT}{\tfT}-\ln(\frac{\fT}{\tfT})-1)\right)\pt_t\varphi
		+\ep^2\frac{\rho}{\tfT}\bfpsi\cdot\pt_t\bfpsi
		+\frac{\rho}{\gamma-1}(\frac{1}{\tfT}-\frac{1}{\fT})\pt_t\theta.   \label{pteta}
	\end{equation}
	It follows from $\eqref{eq}_1$ and integration by parts that
	\begin{equation}
		\begin{aligned}
			&\int \left(\ln(\frac{\rho}{\trho})+\frac{\ep^2}{2}\frac{|\bfpsi|^2}{\tfT}
			+\frac{1}{\gamma-1}(\frac{\fT}{\tfT}-\ln(\frac{\fT}{\tfT})-1)\right)
			\pt_t\varphi \md\bfx\\
			=&\int -\left(\ln\rho+\ln\tfT+\frac{\ep^2}{2}\frac{|\bfpsi|^2}{\tfT}
			+\frac{1}{\gamma-1}(\frac{\fT}{\tfT}-\ln\fT+\ln\tfT)\right)\div(\rho u) \md\bfx\\
			=&\int \bfu\cdot\nabla\rho \md\bfx+\int\rho \bfu\cdot\nabla\left(\frac{\ep^2|\bfpsi|^2}{2\tfT}+\frac{1}{\gamma-1}(\frac{\fT}{\tfT}-\ln\fT)\right)\md\bfx
			+\frac{\gamma}{\gamma-1}\int\frac{\rho}{\tfT}\bfpsi\cdot\nabla\tfT \md\bfx. \label{eta1}
		\end{aligned}
	\end{equation}
	Similarly, the equations $\eqref{eq}_2$--$\eqref{eq}_3$, together with integration by parts lead to
	\begin{equation}
		\begin{aligned}
			\int \ep^2\frac{\rho}{\tfT}\bfpsi\cdot\pt_t\bfpsi \md\bfx
			=&-\ep^2\int\big( \frac{\rho}{\tfT} \bfu\cdot\nabla(\frac{|\bfpsi|^2}{2}) +\frac{\rho}{\tfT}(\bfpsi\cdot\nabla\tbfu)\cdot\bfpsi\big)\md\bfx\\
			&+\int\rho\fT(\frac{\div\bfpsi}{\tfT}-\frac{\bfpsi\cdot\nabla\tfT}{\tfT^2}) \md\bfx\\
			&+\ep^2\int\big(\frac{\mu}{\tfT}\Delta\bfpsi\cdot\bfpsi+\frac{(\mu+\mup)}{\tfT}\nabla\div\bfpsi\cdot\bfpsi\big) \md\bfx, \label{eta2}
		\end{aligned}
	\end{equation}
	and
	\begin{equation}
		\begin{aligned}
			\int \rho(\frac{1}{\tfT}-\frac{1}{\fT})\pt_t\theta \md\bfx
			=&\frac{1}{\gamma-1}\int\rho u\cdot\nabla\big(\ln\fT-\frac{\fT}{\tfT}\big)\md\bfx
			-\frac{1}{\gamma-1}\int \frac{\rho\fT}{\tfT^2}u\cdot\nabla\tfT \md\bfx\\
			+&\int\big(\rho\div\bfpsi-\frac{\rho\fT}{\tfT}\div\bfpsi\big)\md\bfx
			+\int\kappa(\frac{1}{\tfT}-\frac{1}{\fT})\Delta\theta \md\bfx  \\
			&+\ep^2\int(\frac{1}{\tfT}-\frac{1}{\fT}) 
			\big(2\mu|\Dpsi|^2+4\mu\Dtbfu:\Dpsi+\mup(\div\bfpsi)^2\big)\md\bfx.\label{eta3}
		\end{aligned}
	\end{equation}
	
	Substituting \eqref{eq} into \eqref{pteta}, and integrating the resulting equation over $\Omega$, we obtain, by using \eqref{eta1}--\eqref{eta3} and some direct algebraic manipulations, that
	\begin{equation}
		\begin{aligned}
			\frac{d}{dt}\int\eta \md\bfx
			=&-\frac{\gamma}{\gamma-1}\int\frac{\rho\theta \bfpsi\cdot\nabla\tfT}{\tfT^2} \md\bfx  
			-\ep^2\int  \big(\frac{1}{2}\frac{\rho|\bfpsi|^2\bfpsi\cdot\nabla\tfT}{\tfT^2}
			+\frac{\rho}{\tfT}(\bfpsi\cdot\nabla\tbfu)\cdot\bfpsi \big)\md\bfx\\
			&-\ep^2\int\big(\frac{\mu}{\fT}|\nabla\bfpsi|^2+\frac{\mu+\mup}{\fT}(\div\bfpsi)^2\big) \md\bfx
			-\ep^2\int \frac{\mu(\bfpsi\cdot\nabla\bfpsi)\cdot\nabla\tfT}{\tfT^2}\md\bfx\\
			&-\ep^2\int\frac{\mup(\nabla\tfT\cdot\bfpsi)\div\bfpsi}{\tfT^2}\md\bfx
			-\int\frac{\kappa}{\fT^2}|\nabla\theta|^2\md\bfx
			+\int \frac{\kappa\theta(\tfT+\fT)\nabla\tfT\cdot\nabla\theta}{\tfT^2\fT^2}\md\bfx\\
			&+\int \frac{4\mu\ep^2\theta}{\tfT\fT}\Dtbfu:\Dpsi \md\bfx  \\
			:=& \sum_{i=1}^{8}I_i,    \label{I0}
		\end{aligned}
	\end{equation}
	where we have used the facts that
	\begin{equation}
		2\int |\Dpsi|^2\md\bfx=\int\big(|\nabla\bfpsi|^2+(\div\bfpsi)^2\big)\md\bfx, 
	\end{equation}
    and
	\begin{equation}
		\begin{aligned}
				&\int \bfu\cdot\nabla\rho \md\bfx + \int \rho\div\bfpsi \md\bfx =
			\int \bfu\cdot\nabla\rho \md\bfx + \int \rho\div\bfpsi \md\bfx + \int \rho\div \tbfu \md\bfx \\
			=&\int \div(\rho \bfu) \md\bfx=\int\div(\trho\tbfu+\trho\bfpsi+\varphi\tbfu+\varphi\bfpsi)\md\bfx
			=\int\div(\trho\tbfu) \md\bfx\\
			=&\lim_{R\rightarrow\infty}\int_{0}^{1}\int_{-R}^{R}\div(\trho\tbfu)\md\bfx_1\md\bfx_2
			=0.
		\end{aligned}
	\end{equation}
	
	We estimate all the terms on the right-hand side of \eqref{I0} as follows. It follows from \eqref{ReMAnew}, \eqref{ob-1} and \Cref{Poin} that
	\begin{equation}
		I_3+I_6\leq -\frac{c_1}{\Re}\ep^2\|\bfpsi\|^2_{H^1}-\frac{c_2}{\Re}\|\theta\|^2_{H^1},  \notag
	\end{equation}
	where $c_1$ and $c_2$ are positive constants independent of $\ep$, $\Re$ and $T$. Recalling the fact from \eqref{def:Couette} and \eqref{chi} that
	\begin{equation}
		\|\nabla\tfT\|_{L^\infty}=\|(1-\chi)-\frac{\ep^2\Pr }{2C_p}(2x_2-1)\|_{L^\infty}\leq \tC\ep, \notag
	\end{equation}
	we use \Cref{Poin} and \eqref{ob-0} to obtain that
	\begin{equation}
		I_1\leq \tC\|\nabla\tfT\|_{L^\infty}\|\bfpsi\|_{H^1}\|\theta\|_{H^1}
		\leq \tC\|\nabla\tfT\|^2_{L^\infty}\|\bfpsi\|^2_{H^1}+\tC\|\theta\|^2_{H^1}
		\leq \tC\ep^2\|\bfpsi\|^2_{H^1}+\tC\|\theta\|^2_{H^1}.  \notag
	\end{equation}
	Similarly, we have
	\begin{equation}
		I_2 \leq \ep^2\tC\|\bfpsi\|^2_{H^1}, 
		\quad 	I_4+I_5\leq  \frac{\tC}{\Re}\ep^3\|\bfpsi\|^2_{H^1}, \quad I_7\leq \frac{\tC\ep}{\Re}\|\theta\|^2_{H^1},  \notag
	\end{equation}
	and
	\begin{equation}
		I_8\leq \frac{\tC}{\Re}\ep^2\|\theta \|_{L^2}\|\nabla\bfpsi\|_{L^2}
		\leq \frac{\tC}{\Re}(\ep^3\|\bfpsi\|^2_{H^1}+\ep\|\theta\|^2_{H^1}), \notag
	\end{equation}
	where we have used the Young inequality. 
	
	Substituting the estimates of $I_1$ through $I_8$ into \eqref{I0}, and integrating the resulting inequality over the time interval $[0,t]$, we get
	\begin{equation}
		\int \eta(t)\md\bfx+C^{-1}\int_{0}^{t}\big(\ep^2\|\bfpsi(s)\|^2_{H^1}+\|\theta(s)\|^2_{H^1}\big)\md s
		\leq \int\eta(0)\md\bfx,  	\quad\forall t\in[0,T], \label{equiv3}
	\end{equation}
	provided that both $\Re$ and $\ep$ are small enough.
	
	Finally, \eqref{equiv2} and \eqref{equiv3} imply \eqref{eneA0}. The proof is completed.
\end{proof}

\subsection{Estimates for the first-order derivatives.}
We proceed to establishing the \textit{a priori} $H^1$-type estimates. 

We first derive the estimates for $\sup_{t\in [0,T]}A_1(t)$ and $\int_{0}^{T}\|\pt_t \varphi(t)\|^2_{L^2}\md t$.
\begin{lemma} \label{lemma1}
	Suppose that \eqref{chi} holds, and suppose that $\Re<\Re_0$ and $\ep<\vep_1$ as in \Cref{basic}. Then, the following estimates hold:
	\begin{equation} \label{ene1}	
		\int_{0}^{t}\|\pt_t \varphi(s)\|^2_{L^2}\md s
		\leq C\int_{0}^{t}\big(\|\bfpsi(s)\|^2_{H^1}
		+\|\pt_{1}\varphi(s)\|^2_{L^2}\big)\md s,
		\quad\forall t\in[0,T],   
	\end{equation}
	and
	\begin{equation}\label{ene2} 
		A_1(t)\leq C \big(A_1(0)+\ep^{-2}A_0(0)\big)
		+C\ep^{-2}\int_{0}^{t}\|\pt_1\varphi\|^2_{L^2}\md s,
		\quad\forall t\in[0,T],
	\end{equation}
	where the definitions of $A_1$ and $A_0$ are given in \eqref{def:A0123}.
\end{lemma}

\begin{proof}
	It follows from $\eqref{eq}_1$, \eqref{ob-1} and \Cref{Poin} that
	\begin{equation}
		\begin{aligned}
			\|\pt_t\varphi\|_{L^2}\leq& \|\tbfu\cdot\nabla\varphi\|_{L^2}
			+(\|\nabla\varphi\|_{L^\infty}+\|\nabla\trho\|_{L^\infty})\|\bfpsi\|_{L^2}+\|\rho\|_{L^\infty}\|\nabla\bfpsi\|_{L^2}\\
			\leq&\|\pt_{1}\varphi\|_{L^2}+ C\|\bfpsi\|_{H^1},    \label{pt1}
		\end{aligned}
	\end{equation}
	where the fact that $\tu^2=0$ has been used. Integrating \eqref{pt1} over the time interval $[0,t]$ gives \eqref{ene1}.
	
	Next, multiplying $\eqref{eq}_3$ by $\pt_t \theta$, we obtain, by using the method of integration by parts, that
	\begin{equation}
		\begin{aligned}
			&\frac{d}{dt}(\frac{\kappa}{2}\|\nabla\theta\|^2_{L^2})+\int\frac{\rho}{\gamma-1} (\pt_t\theta)^2\md\bfx\\
			\leq&\delta\|\pt_t\theta\|^2_{L^2}
			+C(1+\frac{1}{\delta})\big(\|\bfpsi\|^2_{H^1}+\|\theta\|^2_{H^1}
			+\ep^2(1+\|\nabla\bfpsi\|^2_{L^\infty})\|\nabla\bfpsi\|^2_{L^2}\big)\\
			\leq &\delta\|\pt_t\theta\|^2_{L^2}+C(1+\frac{1}{\delta})\big(\|\bfpsi\|^2_{H^1}+\|\theta\|^2_{H^1}\big),
			\quad\quad \forall\,\delta>0, \label{pt2}
		\end{aligned}
	\end{equation}
	where we have used the Young inequality in the first inequality, and we have used \eqref{ob-1} in the second inequality. Integrating \eqref{pt2} over the time interval $[0,t]$ and choosing $\delta$ suitably small, we get
	\begin{equation}
		\begin{aligned}
			&\|\nabla \theta(t)\|^2_{L^2}+\int_{0}^{t}\|\pt_t \theta(s)\|^2_{L^2}\md s \\
			\leq& C\|\nabla \theta_0\|^2_{L^2}
			+C\int_{0}^{t}\big(\|\bfpsi(s)\|^2_{H^1}+\|\theta(s)\|^2_{H^1}\big)\md s,\quad\forall t\in[0,T].   \label{tempene}
		\end{aligned}
	\end{equation}
	Similarly, multiplying $\pt_t \bfpsi$ on $\eqref{eq}_2$, we obtain, by using the method of integration by parts, that
	\begin{equation}
		\begin{aligned}
			&\frac{d}{dt}\big(\frac{\mu}{2}\|\nabla\bfpsi\|^2_{L^2}+\frac{\mu+\mup}{2}\|\div\bfpsi\|^2_{L^2}\big)
			+\int\rho (\pt_t\bfpsi)^2\md\bfx\\
			\leq & \delta\|\pt_t\bfpsi\|^2_{L^2}+C(1+\frac{1}{\delta})\big(\|\bfpsi\|^2_{H^1}+\ep^{-2}\int\nabla P(\rho,\fT)\pt_t\bfpsi \md\bfx\big), \quad \forall\,\delta>0.     \label{pt3}
		\end{aligned}
	\end{equation}
	where we have used \eqref{ob-1} and the Young inequality.
	Note that 
	\begin{equation} 
		\begin{aligned}
			&\ep^{-2}\int\nabla P(\rho,\tfT)\pt_t\bfpsi \md\bfx\\
			=&\ep^{-2}\int\nabla \big(P(\rho,\fT)-P(\trho,\tfT)\big)\pt_t\bfpsi \md\bfx
			=-\ep^{-2}\int (\rho\fT-\trho\tfT)\pt_t\div\bfpsi \md\bfx\\
			=&-\ep^{-2}\pt_t\big(\int(\rho\fT-\trho\tfT)\div\bfpsi \md\bfx\big)
			+\ep^{-2}\int\pt_t(\rho\fT)\div\bfpsi \md\bfx\\
			=&-\ep^{-2}\pt_t\big(\int(\fT\varphi+\trho\theta)\div\bfpsi \md\bfx\big)
			+\ep^{-2}\int(\fT\pt_t\varphi+\rho\pt_t\theta)\div\bfpsi \md\bfx\\
			\leq&-\ep^{-2}\pt_t\big(\int(\tfT\varphi+\theta\varphi+\trho\theta)\div\bfpsi \md\bfx\big)
			+C\ep^{-2}\big(\|\pt_t\varphi\|^2_{L^2}
			+\|\pt_t\theta\|^2_{L^2}+\|\bfpsi\|^2_{H^1}\big).      \label{pt4}
		\end{aligned}
	\end{equation}
	The Young inequality gives 
	\begin{equation}\label{pt5}
		\begin{aligned}
		&\ep^{-2}\int(\tfT\varphi+\theta\varphi+\trho\theta)\div\bfpsi \md\bfx
		\leq \delta\|\nabla\bfpsi\|^2_{L^2}
		+C(1+\frac{1}{\delta})\ep^{-4}\big(\|\varphi\|^2_{L^2}+\|\theta\|^2_{L^2}\big),
		\quad\forall \delta>0,     \\    
		&\ep^{-2}\int(\tfT\varphi_0+\theta_0\varphi_0 +\trho\theta_0)\div\bfpsi_0 \md\bfx
		\leq C\|\nabla\bfpsi_0\|^2_{L^2}
		+C\ep^{-4}\big(\|\varphi_0\|^2_{L^2}+\|\theta_0\|^2_{L^2}\big).        
		\end{aligned}
	\end{equation}
	Therefore, based on \eqref{pt4}--\eqref{pt5}, integrating \eqref{pt3} over the time interval $[0,t]$, and choosing $\delta$ small enough, we have
	\begin{equation}
		\begin{aligned}
			&\|\nabla \bfpsi(t)\|^2_{L^2}+\int_{0}^{t}\|\pt_t \bfpsi(s)\|^2_{L^2}\md s\\
			\leq &C\ep^{-4}\big(\|(\varphi,\theta)(t)\|^2_{L^2}
			+\|(\varphi_0,\theta_0)\|^2_{L^2}\big)+C\|\nabla \bfpsi_0\|^2_{L^2}\\
			&+C\ep^{-2}\int_{0}^{t}\big(\|\pt_t\varphi\|^2_{L^2}
			+\|\pt_t\theta\|^2_{L^2}+\|\bfpsi\|^2_{H^1}\big)\md s,
			\quad\forall t\in[0,T].\label{tempene3} 
		\end{aligned}
	\end{equation}
	
	Finally, adding \eqref{tempene} multiplied by $2C\ep^{-2}$ to \eqref{tempene3} derives
	\begin{equation}
		\begin{aligned}
			A_1(t)\leq& CA_1(0)+C\ep^{-2}A_0(t)+C\ep^{-2}\int_{0}^{t}\|\pt_t\varphi\|^2_{L^2}\md s\\
			\leq& CA_1(0)+C\ep^{-2}A_0(0)+C\ep^{-2}\int_{0}^{t}\|\pt_1\varphi\|^2_{L^2}\md s,
			\quad\forall t\in[0,T],
		\end{aligned}
	\end{equation}
	where \eqref{eneA0} and \eqref{ene1} have been used in the second inequality.
	
	The proof is completed.
\end{proof}

Next, the equations $\eqref{eq}_1$ and $\eqref{eq}_2$ can be rewritten as
\begin{equation}
	\pt_t\varphi+\bfu\cdot\nabla\varphi+\div\bfpsi=f_1,   \label{eqf1}
\end{equation}
and
\begin{equation}
	\rho\big(\pt_t\bfpsi+\bfu\cdot\nabla\bfpsi+\bfpsi\cdot\nabla\tbfu\big)
	+\ep^{-2}\big(\nabla\varphi+\nabla\theta\big)
	-\mu\Delta\bfpsi-(\mu+\mup)\nabla\div\bfpsi=\ep^{-2}\nabla f_2,\label{eqf2}
\end{equation}
where 
\begin{equation}
	f_1=-(\rho-1)\div\bfpsi-\bfpsi\cdot\nabla\trho, \quad\quad
	f_2=-\big((\trho-1)\theta+(\tfT-1)\varphi+\varphi\theta\big). \notag
\end{equation}

The estimates for $\sup_{t\in [0,T]}\ep^{-2}\|\nabla\varphi(t)\|^2_{L^2}$
and $\int_{0}^{T}\|\nabla\div \bfpsi(t)\|^2_{L^2}\md t$ are derived as follows.  

\begin{lemma}  \label{lemma2}
	Suppose that \eqref{chi} holds, and suppose that $\Re<\Re_0$ and $\ep<\vep_1$ as in \Cref{basic}. Then, it holds that
\begin{equation} \label{ene3}
	\begin{aligned}
		&\ep^{-2}\|\pt_{1}\varphi(t)\|^2_{L^2}+\|\pt_{1}\bfpsi(t)\|^2_{L^2}
		+\int_{0}^{t}\big(\|\nabla\pt_{1}\bfpsi(s)\|^2_{L^2}
		+\|\pt_{1}\div\bfpsi(s)\|^2_{L^2}\big)\md s\\	
		\leq& C\big(\ep^{-2}\|\pt_{1}\varphi_0\|^2_{L^2}+\|\pt_{1}\bfpsi_0\|^2_{L^2}\big)
		+C\ep^{-2}\int_{0}^{t} \|\pt_{1}\varphi(s)\|^2_{L^2}\md s	\\
		&+C\int_{0}^{t}\big(\ep^{-4}\|\theta(s)\|^2_{H^1}+\ep^{-2}\|\bfpsi(s)\|^2_{H^1}		+\|\pt_t\bfpsi(s)\|^2_{L^2}\big)\md s,  \quad\quad\forall t\in [0,T].
	\end{aligned}
\end{equation}
\end{lemma}
\begin{proof}
	Noting that $\pt_{1}\trho=0$ and $\pt_{1}\tu=0$, we first differentiate both $\eqref{eq}_1$ and $\eqref{eq}_2$ in $x_1$ to get
	\begin{equation}
		\pt_t\pt_{1}\varphi+\bfu\cdot\nabla\pt_{1}\varphi+\div(\pt_{1}\bfpsi)
		=\pt_{1}f_1-\pt_{1}\bfpsi\cdot\nabla\varphi, \label{pertur6} 
	\end{equation}
	and	
	\begin{equation}
		\begin{aligned}
			&\rho(\pt_t\pt_{1}\bfpsi+\bfu\cdot\nabla \pt_{1}\bfpsi)
			+\ep^{-2}\nabla \big(\pt_{1}\varphi+\pt_{1}\theta)-\mu\Delta \pt_{1}\bfpsi-(\mu+\mup)\nabla\pt_{1}\div \bfpsi\\
			=&\nabla\pt_1f_2+\bfR_1+\bfR_2,\label{pertur7}
		\end{aligned}
	\end{equation}
	where
	\begin{equation}
		\quad \bfR_1=-\pt_{1}\varphi(\pt_t\bfpsi+\bfu\cdot\nabla\bfpsi+\bfpsi\cdot\nabla\tbfu),\quad \bfR_2=-\rho\pt_{1}\bfpsi\cdot\nabla(\tbfu+\bfpsi). \notag
	\end{equation}
	It follows from \eqref{ob-1} and \eqref{def:N} that
	\begin{equation} \label{Landou1}
		\begin{aligned}
		&\bfR_1=O(1)|\pt_1\varphi|(|\pt_t\bfpsi|+|\nabla\bfpsi|+|\bfpsi|)
			=O(1)(|\pt_t\bfpsi|+|\nabla\bfpsi|+|\bfpsi|),\\
		&\bfR_2=O(1)|\nabla\bfpsi|(1+|\nabla\bfpsi|)
			=O(\ep^{-2})|\nabla\bfpsi|. 
		\end{aligned}
	\end{equation}
	Throughout this paper, the Landau notation $O(\ep^k)$ is used to indicate a function whose absolute value remains uniformly bounded by $C\ep^k$ for a integer $k$ and a positive constant $C$ independent of $\ep$ and $T$. Similarly, we have 
	\begin{equation}\label{Landou2}
		\begin{aligned}
		\pt_{1}f_1=&O(1)(|\trho-1|+|\varphi|)|\pt_{1}\div\bfpsi|
	+O(1)|\pt_1\varphi||\div\bfpsi|+O(1)|\nabla\trho||\pt_{1}\bfpsi|\\
		=&O(\ep)|\pt_{1}\div\bfpsi| + O(1)|\nabla\bfpsi|,\\
		\pt_{1}f_2=&O(1)(|\trho-1|+|\tfT-1|
		+|\varphi|+|\theta|)(|\pt_1\varphi|+|\pt_1\theta|)\\
		=&O(\ep)(|\pt_1\varphi|+|\pt_1\theta|).
			\end{aligned}
	\end{equation}
    It is derived from \eqref{Landou1}--\eqref{Landou2} that
    \begin{equation}\label{est-Landou}
    	\begin{aligned}
    			&\|\bfR_1\|^2_{L^2}+\|\bfR_2\|^2_{L^2}
    		\leq C\ep^{-2}\|\bfpsi\|^2_{H^1}
    		+C\ep^2\|\pt_{t}\bfpsi\|^2_{L^2},\\
    		&\|\pt_{1}f_1\|^2_{L^2}
    		\leq C\ep^2\|\pt_1\nabla\bfpsi\|^2_{L^2}
    		+C\|\bfpsi\|^2_{H^1}, \quad
    		\|\pt_{1}f_2\|^2_{L^2}
    		\leq C\ep^2\|\pt_1\varphi\|^2_{L^2}
    		+C\ep^2\|\theta\|^2_{H^1}.
    	\end{aligned}
    \end{equation}
	Thus, based on \eqref{ob-0} and \eqref{ob-1}, adding \eqref{pertur6} multiplied by $\ep^{-2}\pt_{1}\varphi$ to \eqref{pertur7} multiplied by $\pt_{1}\bfpsi$, we obtain, by applying the method of integration by parts and using the Young inequality, that
	\begin{equation}
		\begin{aligned}
			&\frac{\md}{\md t}
			\big(\frac{1}{2}\int\ep^{-2}|\pt_{1}\varphi|^2+\rho|\pt_{1}\bfpsi|^2\md\bfx\big)
			+\mu\|\nabla\pt_{1}\bfpsi\|^2_{L^2}+(\mu+\mup)\|\pt_{1}\div\bfpsi\|^2_{L^2}    \\
			\leq&  \frac{1}{2} \ep^{-2}\|\pt_{1}\varphi\|_{L^\infty}\|\div\bfpsi\|_{L^2}\|\pt_{1}\varphi\|_{L^2}
			+\ep^{-2}\|\pt_1f_1\|_{L^2}\|\pt_1\varphi\|_{L^2}\\
			&+\ep^{-2}\|\nabla\varphi\|_{L^\infty}\|\pt_1\bfpsi\|_{L^2}\|\pt_1\varphi\|_{L^2} 
			+\ep^{-2}\big(\|\pt_1\theta\|_{L^2}+\|\pt_1f_2\|_{L^2}\big)\|\div\pt_1\bfpsi\|_{L^2}\\
			&+\big(\|\bfR_1\|_{L^2}+\|\bfR_2\|_{L^2}\big)\|\pt_1\bfpsi\|_{L^2}
			\\
			\leq& \frac{\mu}{4}\|\nabla\pt_{1}\bfpsi\|^2_{L^2}+C(1+\frac{1}{\delta})\ep^{-2}\|\pt_1\varphi\|^2_{L^2}
			+\delta\ep^{-2}\|\pt_1f_1\|^2_{L^2}	\\
			&+C\big(\ep^{-4}\|\pt_1\theta\|^2_{L^2}+\ep^{-4}\|\pt_1f_2\|^2_{L^2}
			+\|\bfR_1\|^2_{L^2}+\|\bfR_2\|^2_{L^2}+\ep^{-2}\|\nabla\bfpsi\|^2_{L^2}\big), \quad \forall \delta>0. 
				\end{aligned}
		\end{equation}
		Substitute \eqref{est-Landou} into the above inequality and  let $\delta$ be a small positive constant satisfying $C\delta\leq \frac{\mu}{4}$, yielding
			\begin{equation}
				\begin{aligned}
		&\frac{\md}{\md t}
	\big(\frac{1}{2}\int\ep^{-2}|\pt_{1}\varphi|^2+\rho|\pt_{1}\bfpsi|^2\md\bfx\big)
	+\frac{\mu}{2}\|\nabla\pt_{1}\bfpsi\|^2_{L^2}+(\mu+\mup)\|\pt_{1}\div\bfpsi\|^2_{L^2} \\
			\leq&  		C\ep^{-2}\|\pt_1\varphi\|^2_{L^2}
			+C\big(\ep^{-4}\|\theta\|^2_{H^1}
			+\|\pt_t\bfpsi\|^2_{L^2}+\ep^{-2}\|\bfpsi\|^2_{H^1}\big),    \quad \quad \forall \delta>0.   \label{div1}
		\end{aligned}
	\end{equation}
	Integrating \eqref{div1} over the time interval $[0,t]$ leads to \eqref{ene3}. 

The proof is completed.
\end{proof}

	For convenience, in what follows we adopt the abbreviated notations
	\begin{equation} \label{def:G}
		G_1:=\tfT\pt_{2}\varphi+\varphi\pt_{2}\tfT,\quad 	G_2:=\theta\pt_{2}\varphi+\varphi\pt_{2}\theta+\theta\pt_{2}\trho+\trho\pt_{2}\theta.  
	\end{equation}	
	The estimate for $\ep^{-4}\int_{0}^{t}\|G_1\|^2_{L^2}\md s$ is obtained as follows.
	\begin{lemma}\label{lemma3}
			Suppose that \eqref{chi} holds, and suppose that $\Re<\Re_0$ and $\ep<\vep_1$ as in \Cref{basic}. Then, it holds that
		\begin{equation}\label{ene4}
		\begin{aligned}
			&\ep^{-2}\|\pt_{2}\varphi(t)\|^2_{L^2}
			+\int_{0}^{t}\|\pt_{2}\div\bfpsi(s)\|^2_{L^2}\md s
			+ \int_{0}^{t}\ep^{-4}\|G_1(s)\|^2_{L^2}\md s\\
			\leq& C\ep^{-2}\big(\|\varphi_0\|^2_{H^1}+\|\varphi(t)\|^2_{L^2}\big)
			+C_1\int_{0}^{t}\|\nabla\pt_{1}\bfpsi(s)\|^2_{L^2}\md s\\
			&+  C\int_{0}^{t}\bigg(
			\|\pt_t\bfpsi(s)\|^2_{L^2}
			+\|\bfpsi(s)\|^2_{H^1}+\ep^{-4}\|\theta(s)\|^2_{H^1}
			+\ep^{2}\|\pt_t\varphi(s)\|^2_{L^2}	\bigg)\md s,
			\quad \forall t\in[0,T], 
		\end{aligned}
	\end{equation}
	where $C_1$ is a positive constant independent of $\ep$ and $T$.
	\end{lemma}
	
	\begin{proof}
	We first obtain from a direct calculation that
	\begin{equation}
		\mu\Delta \psi^2+(\mu+\mup)\pt_{2}\div \bfpsi
		=(2\mu+\mup)\pt^2_{2}\psi^2+\mu\pt_{1}(\pt_{1}\psi^2-\pt_{2}\psi^1)
		+(\mu+\mup)\pt_{2}\pt_{1}\psi^1,\label{eq:normal1}
	\end{equation}
	and
	\begin{equation} \label{eq:normal2}
		\pt_{2} P(\rho,\fT)=\pt_{2}\big(P(\rho,\fT)-P(\trho,\tfT)\big)              
		=G_1+G_2. 
	\end{equation}
	Substituting the above equations \eqref{eq:normal1} and \eqref{eq:normal2} into $\eqref{eq}_2$ gives
	\begin{equation}
		-(2\mu+\mup)\pt^2_{2}\psi^2
		+\ep^{-2}G_1=R_3,\label{eqnormal2}
	\end{equation}
	where
	\begin{equation}\label{def:R3}
			R_3=\mu\pt_{1}(\pt_{1}\psi^2-\pt_{2}\psi^1)+(\mu+\mup)\pt_{2}\pt_{1}\psi^1
			-\rho\big(\pt_t\psi^2+\bfu\cdot\nabla\psi^2+\bfpsi\cdot\nabla\tu^2\big)-\ep^{-2}G_2. 
	\end{equation}

Next, noting that $\pt_{2}\tbfu\cdot\nabla\varphi=\pt_{2}\tu^1\pt_{1}\varphi=\pt_{1}\varphi$,
differentiating the equation $\eqref{eq}_1$ in $x_2$ leads to that
\begin{equation}
	\pt_t\pt_{2}\varphi+\bfu\cdot\nabla\pt_{2}\varphi+\pt_{1}\varphi
	+\rho\pt^2_{2}\psi^2 = R_4,   \label{eqnormal1}
\end{equation}
where 
\begin{equation}\label{def:R4}
		R_4=-\pt_2\bfpsi\cdot\nabla\rho-\bfpsi\cdot\nabla\pt_2\trho-\pt_2\rho\div\bfpsi-\rho\pt_{2}\pt_{1}\psi^1.   
\end{equation}
	
	Adding \eqref{eqnormal1} multiplied by $(2\mu+\mup)\ep^{-2}G_1$ to \eqref{eqnormal2} multiplied by $\ep^{-2}\rho G_1$, we integrate the resulting equation over $\Omega$ to obtain 
	\begin{equation} \label{eq:inteG1}
		\begin{aligned}
				&(2\mu+\mup)\ep^{-2}\int \pt_t\pt_{2}\varphi G_1 \md\bfx
			+\ep^{-4}\int \rho G_1^2 \md\bfx\\
			=& 	(2\mu+\mup)\ep^{-2}\bigg[-\int (\tbfu\cdot\nabla\pt_{2}\varphi) G_1 \md\bfx
			-\int (\bfpsi\cdot\nabla\pt_{2}\varphi) G_1 \md\bfx
			+\int G_1(R_3+R_4)\md\bfx\bigg].
		\end{aligned}
	\end{equation}
	
	The terms of integral on the left-hand side of the equation \eqref{eq:inteG1} are estimated as follows. It follows from \eqref{ob-1} that
	\begin{equation}\label{est:H1-1}
		\ep^{-4}\int \rho G_1^2 \md\bfx 
		\geq \ep^{-4}(\inf_{(x,t)\in \Omega\times[0,T]}\rho)\int  G_1^2 \md\bfx
		\geq \frac{1}{2}\ep^{-4}\|G_1\|^2_{L^2}.
	\end{equation}
	By the method of integration by parts, it holds that
		\begin{equation}\label{eq:temp-H1}
		\begin{aligned}
     &\int \pt_t\pt_{2}\varphi G_1\md\bfx =\int \pt_t\pt_{2}\varphi
     (\tfT\pt_{2}\varphi+\varphi\pt_{2}\tfT)\md\bfx\\
     =&\frac{d}{dt}\int\big( \frac{1}{2}\tfT|\pt_{2}\varphi|^2
     + \varphi\pt_{2}\tfT\pt_{2}\varphi \big)\md\bfx
     -\int\pt_{2}\varphi\pt_{t}\varphi\pt_{2}\tfT \md\bfx\\
     =&\frac{d}{dt}\int\big( \frac{1}{2}\tfT|\pt_{2}\varphi|^2
     + \varphi\pt_{2}\tfT\pt_{2}\varphi \big)\md\bfx
      -\int\frac{\pt_{2}\tfT}{\tfT}(\tfT\pt_{2}\varphi
      +\varphi\pt_{2}\tfT-\varphi\pt_{2}\tfT )\pt_{t}\varphi \md\bfx\\
      =&\frac{d}{dt}\int\big( \frac{1}{2}\tfT|\pt_{2}\varphi|^2
        + \varphi\pt_{2}\tfT\pt_{2}\varphi \big)\md\bfx
      -\int\big(\frac{\pt_{2}\tfT}{\tfT}G_1\pt_{t}\varphi 
      - \frac{\pt_{2}\tfT}{2\tfT}\pt_{t}(\varphi^2)   \big)\md\bfx\\
      =&\frac{d}{dt}\int\big( \frac{1}{2}\tfT|\pt_{2}\varphi|^2
        + \varphi\pt_{2}\tfT\pt_{2}\varphi 
        + \frac{\pt_{2}\tfT}{2\tfT}\varphi^2\big)\md\bfx
        - \int\big(\frac{\pt_{2}\tfT}{\tfT}G_1\pt_{t}\varphi \big) \md\bfx,
			\end{aligned}
		\end{equation}
Noting the fact from \eqref{ob-0} that 
\begin{equation}\label{fact-1}
	\|\frac{\pt_{2}\tfT}{2\tfT} \|_{L^\infty}\leq \tC\epsilon,
\end{equation}
the H\"{o}lder inequality and the Young inequality derive
		\begin{equation}\notag
	\begin{aligned}
		 -(2\mu+\mup)\ep^{-2}\int\big(\frac{\pt_{2}\tfT}{\tfT}G_1\pt_{t}\varphi \big) \md\bfx
	\leq&(2\mu+\mup)\ep^{-2}\|\frac{\pt_{2}\tfT}{\tfT}\|_{L^\infty}\|G_1\|_{L^2}\|\pt_{t}\varphi \|_{L^2}\\
	\leq&\frac{1}{16}\ep^{-4}\|G_1\|^2_{L^2}+ C\ep^2\|\pt_{t}\varphi \|^2_{L^2},
	\end{aligned}
\end{equation}
Substituting the above inequality into \eqref{eq:temp-H1} gives
		\begin{equation}\label{est:H1-2}
		\begin{aligned}
    (2\mu+\mup)\ep^{-2}\int\pt_t\pt_{2}\varphi G_1\md\bfx 
	 \geq &(2\mu+\mup)\ep^{-2}\frac{d}{dt}\int\big( \frac{1}{2}\tfT|\pt_{2}\varphi|^2
	+ \varphi\pt_{2}\tfT\pt_{2}\varphi 
	+ \frac{\pt_{2}\tfT}{2\tfT}\varphi^2\big)\md\bfx\\
	&-\frac{1}{16}\ep^{-4}\|G_1\|^2_{L^2} - C\ep^2\|\pt_{t}\varphi \|^2_{L^2},
\end{aligned}
\end{equation}
 
	For the terms on the right-hand side of \eqref{eq:inteG1}, due to the facts that
	 \begin{equation}\notag
		\div(\tbfu\tfT)=0,\quad\div(\tbfu\pt_{2}\tfT)=0, \quad
		\tbfu\cdot\nabla\varphi=\tu^1\pt_{1}\varphi, \quad \pt_{1}(\tu^1\frac{\pt_{2}\tfT}{\tfT})=0,
	\end{equation} 
    we have
	\begin{equation}\notag
		\begin{aligned}
			&\int \tbfu\cdot\nabla\pt_{2}\varphi G_1 \md\bfx
			=\int \tbfu\cdot\nabla\pt_{2}\varphi  (\tfT\pt_{2}\varphi+\varphi\pt_{2}\tfT)\md\bfx\\
			=&\int \frac{1}{2}\div(\tbfu\tfT)(\pt_{2}\varphi)^2
			+ \div(\tbfu\pt_{2}\tfT)\varphi\pt_{2}\varphi
			+ (\tbfu\cdot\nabla\varphi)\pt_{2}\varphi\pt_{2}\tfT \md\bfx\\
			=&\int \tu^1\pt_{1}\varphi\pt_{2}\varphi\pt_{2}\tfT \md\bfx
			=\int \tu^1\pt_{1}\varphi\frac{\pt_{2}\tfT}{\tfT}
			(\tfT\pt_{2}\varphi+\varphi\pt_{2}\tfT-\varphi\pt_{2}\tfT) \md\bfx\\
			=&\int \tu^1\pt_{1}\varphi\frac{\pt_{2}\tfT}{\tfT}G_1 \md\bfx
			-\int \tu^1\frac{\pt_{2}\tfT}{\tfT}\pt_{1}(\varphi^2) \md\bfx\\
			=&\int \tu^1\pt_{1}\varphi\frac{\pt_{2}\tfT}{\tfT}G_1 \md\bfx.\\
		\end{aligned}
	\end{equation}
	This, together with the H\"{o}lder inequality, the Young inequality and \eqref{fact-1}, implies
	\begin{equation}\label{est:H1-3}
			\begin{aligned}
		 \big|-(2\mu+\mup)\ep^{-2}\int \tbfu\cdot\nabla\pt_{2}\varphi G_1 \md\bfx\big|
		 \leq& (2\mu+\mup)\ep^{-2}\|\tbfu\|_{L^\infty}\| \frac{\pt_{2}\tfT}{\tfT}\|_{L^\infty} \|G_1\|_{L^2}\|\pt_{1}\varphi \|_{L^2}\\
		 \leq&  \frac{1}{16}\ep^{-4}\|G_1\|^2_{L^2}
		  +C\ep^2 \|\pt_{1}\varphi \|^2_{L^2}.
		 	\end{aligned}
	\end{equation}
 
 Similarly, it holds that
  \begin{equation}\label{right-term-1}
  	\begin{aligned}
  		&\int \bfpsi\cdot\nabla\pt_{2}\varphi G_1 \md\bfx\\
  	=&\int \bfpsi\cdot\nabla(\frac{G_1}{\tfT}-\frac{\pt_{2}\tfT}{\tfT}\varphi)G_1 \md\bfx\\
  	=&\int \tfT\bfpsi\cdot\nabla(\frac{G_1}{\tfT})\frac{G_1}{\tfT} \md\bfx
  	-\int \bfpsi\cdot\nabla(\frac{\pt_{2}\tfT}{\tfT}\varphi)G_1 \md\bfx\\
  	=&-\frac{1}{2}\int \div(\tfT\bfpsi)(\frac{G_1}{\tfT})^2 \md\bfx
  	-\int \bfpsi\cdot\nabla(\frac{\pt_{2}\tfT}{\tfT}\varphi)G_1 \md\bfx.
    \end{aligned}
  \end{equation}
Due to facts from \eqref{ob-1} that
\begin{equation}\notag
	\begin{aligned}
	   &\|G_1\|_{L^\infty}=\|\tfT\pt_{2}\varphi+\varphi\pt_{2}\tfT\|_{L^\infty}
		\leq \|\nabla\varphi\|_{L^\infty}\|\tfT\|_{L^\infty}
		+\|\nabla\tfT\|_{L^\infty}\|\varphi\|_{L^\infty}
		\leq C,\\
		 &\|\nabla(\frac{\pt_{2}\tfT}{\tfT}\varphi)\|_{L^\infty}
		 \leq \|\nabla(\frac{\pt_{2}\tfT}{\tfT})\|_{L^\infty} \|\varphi\|_{L^\infty}
		 +\|\nabla\varphi\|_{L^\infty}\|\frac{\pt_{2}\tfT}{\tfT}\|_{L^\infty}
		 \leq C,
	\end{aligned}
\end{equation}
we have
 \begin{equation}\label{right-term-2}
	\begin{aligned}
\big|\frac{1}{2}(2\mu+\mup)\ep^{-2}\int\div(\tfT\bfpsi)(\frac{G_1}{\tfT})^2\md\bfx\big|
 \leq& C\ep^{-2}\|\frac{ \div(\tfT\bfpsi)}{\tfT^2}\|_{L^2}\|G_1\|_{L^\infty}\|G_1\|_{L^2}\\
 \leq & \frac{1}{16}\ep^{-4}\|G_1\|^2_{L^2} +C\ep^2 \|\bfpsi \|^2_{H^1}.
	\end{aligned}
\end{equation}
and
\begin{equation}\label{right-term-3}
	\begin{aligned}
	\big|(2\mu+\mup)\ep^{-2}\int\bfpsi\cdot\nabla(\frac{\pt_{2}\tfT}{\tfT}\varphi)G_1\md\bfx\big|
		\leq& C\ep^{-2}\|\bfpsi\|_{L^2}
		\|\nabla(\frac{\pt_{2}\tfT}{\tfT}\varphi)\|_{L^\infty}\|G_1\|_{L^2}\\
		\leq & \frac{1}{16}\ep^{-4}\|G_1\|^2_{L^2} +C\ep^2 \|\bfpsi \|^2_{H^1}.
	\end{aligned}
\end{equation}
It follows from \eqref{right-term-1}--\eqref{right-term-3} that
\begin{equation}\label{est:H1-4}
	\big|-(2\mu+\mup)\ep^{-2}\int (\bfpsi\cdot\nabla\pt_{2}\varphi) G_1 \md\bfx\big|
	\leq  \frac{2}{16}\ep^{-4}\|G_1\|^2_{L^2} +C\ep^2 \|\bfpsi \|^2_{H^1}.
\end{equation}	
	
Moreover, we obtain from \eqref{ob-0}, \eqref{ob-1}, \eqref{def:R3} and \eqref{def:R4} that
		\begin{equation}\label{eq:G2}
		\begin{aligned}
			G_2=&O(1)\big(|\theta|+|\nabla\theta|\big),\\
			R_3=&O(1)\big(|\nabla\pt_{1}\bfpsi|+|\pt_t\bfpsi|
			+|\nabla\bfpsi|+|\bfpsi|\big)
			+O(\ep^{-2})\big(|\theta|+|\nabla\theta|\big),\\
			R_4=&O(1)\big(|\nabla \bfpsi|+|\bfpsi|
			+|\nabla\pt_{1}\bfpsi|\big), 
		\end{aligned}
	\end{equation}
which, together with the H\"{o}lder inequality and the Young inequality, implies
\begin{equation}\label{est:H1-5}
	\begin{aligned}
		&(2\mu+\mup)\ep^{-2}\int G_1(R_3+R_4)\md\bfx\\
		\leq& \frac{1}{16}\ep^{-4} \|G_1\|_{L^2}^2
		+C(\|R_3\|_{L^2}^2+\|R_4\|_{L^2}^2)\\
		\leq& \frac{1}{16}\ep^{-4}\|G_1\|_{L^2}^2
		+C\big(\|\nabla\pt_{1}\bfpsi\|^2_{L^2}+\|\pt_t\bfpsi\|^2_{L^2}
		+\|\bfpsi\|^2_{H^1}\big)+C\ep^{-4}\|\theta\|^2_{H^1}.
	\end{aligned}
\end{equation}
	
Substituting the estimates \eqref{est:H1-1}, \eqref{est:H1-2}, \eqref{est:H1-3}, \eqref{est:H1-4} and \eqref{est:H1-5} into \eqref{eq:inteG1} leads to
\begin{equation}
	\begin{aligned}
		&\ep^{-2}\frac{\md}{\md t}\int\big( \frac{1}{2}\tfT|\pt_{2}\varphi|^2
		+ \varphi\pt_{2}\tfT\pt_{2}\varphi + \frac{\pt_{2}\tfT}{2\tfT}\varphi^2\big)\md\bfx
		+\ep^{-4}\|G_1\|_{L^2}^2\\
		\leq &C\ep^2 \|\pt_{t}\varphi \|^2_{L^2}	+C\big(\|\nabla\pt_{1}\bfpsi\|^2_{L^2}+\|\pt_t\bfpsi\|^2_{L^2}
		+\|\bfpsi\|^2_{H^1}\big)+C\ep^{-4}\|\theta\|^2_{H^1}.
	\end{aligned}
\end{equation}
Integrating the above inequality over the time interval $[0,t]$, we have
	\begin{equation} \label{div3}
		\begin{aligned}
			&\ep^{-2}\|\pt_{2}\varphi(t)\|^2_{L^2}
			+\int_{0}^{t}\ep^{-4}\|G_1(s)\|^2_{L^2}\md s\\
			\leq &C\int_{0}^{t}\bigg(\|\nabla\pt_{1}\bfpsi(s)\|^2_{L^2}
			+\|\pt_t\bfpsi(s)\|^2_{L^2}
			+\|\bfpsi(s)\|^2_{H^1}+\ep^{-4}\|\theta(s)\|^2_{H^1}
			+\ep^{2}\|\pt_t\varphi(s)\|^2_{L^2}	\bigg)\md s\\
		&+ C\ep^{-2}\big(\|\varphi_0\|^2_{H^1}+\|\varphi(t)\|^2_{L^2}\big),
		\quad\quad\quad \forall t\in[0,T], 
		\end{aligned}
	\end{equation}
	where we have used the fact from \eqref{ob-1} and the Young inequality that
	\begin{equation}
		\begin{aligned}\notag
			&\ep^{-2}\int\big(\frac{1}{2}\tfT|\pt_{2}\varphi|^2
			+ \varphi\pt_{2}\tfT\pt_{2}\varphi + \frac{\pt_{2}\tfT}{2\tfT}\varphi^2\big) \md\bfx\\
			\geq& \frac{\ep^{-2}}{2}\big(\inf_{x\in \Omega}\tfT(x)\big)
			\|\pt_{2}\varphi\|^2_{L^2}
			-\ep^{-2}\|\pt_{2}\tfT\|_{L^\infty}\|\varphi\|_{L^2}\|\pt_{2}\varphi\|_{L^2}
			-\frac{\ep^{-2}}{2}\|\frac{\pt_{2}\tfT}{\tfT}\|_{L^\infty}\|\varphi\|^2_{L^2}\\
				\geq& \frac{3}{8}\ep^{-2}\|\pt_{2}\varphi\|^2_{L^2}
			-\frac{1}{8}\ep^{-2}\|\pt_{2}\varphi\|^2_{L^2}
			-C\ep^{-2}\|\pt_{2}\tfT\|^2_{L^\infty}\|\varphi\|^2_{L^2}
			-\frac{\ep^{-2}}{2}\|\frac{\pt_{2}\tfT}{\tfT}\|_{L^\infty}\|\varphi\|^2_{L^2}\\
			\geq& \frac{1}{4}\ep^{-2}\|\pt_{2}\varphi\|^2_{L^2}-C\ep^{-2}\|\varphi\|^2_{L^2}.
		\end{aligned}
	\end{equation}
	
	Finally, it follows from \eqref{eqnormal2} that
	\begin{equation}\label{eq:pt2divpsi}
			(2\mu+\mup)\pt_2\div\bfpsi=(2\mu+\mup)(\pt_2^2\psi^2+\pt_2\pt_1\psi^1)
			=-R_3+\ep^{-2}G_1+(2\mu+\mup)\pt_2\pt_1\psi^1.
	\end{equation}
    This leads to
	\begin{equation}\label{eqnormal3}
		(2\mu+\mup)\pt_{2}\div\bfpsi
		=\ep^{-2}G_1+O(1)\big(|\pt_t\bfpsi|+|\nabla\bfpsi|+|\bfpsi|
		+|\nabla\pt_{1}\bfpsi|\big)+O(\ep^{-2})(|\theta|+|\nabla\theta|),
	\end{equation}
	which implies
	\begin{equation}\notag
			\begin{aligned}
			&\frac{1}{2}(2\mu+\mup)^2\int_{0}^{t}\|\pt_{2}\div\bfpsi(s)\|^2_{L^2}\md s\\
			\leq& \frac{1}{2}\int_{0}^{t}\ep^{-4}\|G_1(s)\|^2_{L^2}\md s
			+C\int_{0}^{t}\bigg(\|\nabla\pt_{1}\bfpsi(s)\|^2_{L^2}
			+\|\pt_t\bfpsi(s)\|^2_{L^2}
			+\|\bfpsi(s)\|^2_{H^1}+\ep^{-4}\|\theta(s)\|^2_{H^1}\bigg)\md s.
				\end{aligned}
	\end{equation}
	This, together with \eqref{div3}, gives \eqref{ene4}.
	
	The proof is completed.
\end{proof}

As a consequence of \Cref{lemma1,lemma2,lemma3}, we reach the estimate for $\int_{0}^{t}\|\nabla\div \bfpsi(s)\|^2_{L^2}\md s$.
\begin{corollary}
		Suppose that \eqref{chi} holds, and suppose that $\Re<\Re_0$ and $\ep<\vep_1$ as in \Cref{basic}. Then, it holds that

	\begin{equation} \label{est-div}
			\begin{aligned}
		&\ep^{-2}\|\nabla\varphi(t)\|^2_{L^2}
	+\int_{0}^{t}\big(\|\nabla\div \bfpsi(s)\|^2_{L^2}
	+\ep^{-4}\|G_1(s)\|^2_{L^2}\big)\md s\\
	\leq& 	C\int_{0}^{t}\ep^{-2}\|\pt_{1}\varphi(s)\|^2_{L^2}\md s 
	+C\big(\ep^{-2}A_0(0)+A_1(0)+A_2(0)\big),  	\quad\quad  \forall\, t\in[0,T],
		\end{aligned}
	\end{equation}
where the definitions of $G_1$ and $A_0$, $A_1$ and $A_2$ are given in \eqref{def:G} and \eqref{def:A0123}, respectively.
\end{corollary} 

\begin{proof}
	Multiplying \eqref{ene3} by $2C_1$ and then adding the resulting inequality to \eqref{ene4}, we have
	\begin{equation}
		\begin{aligned}\notag
			&\ep^{-2}\|\nabla\varphi(t)\|^2_{L^2}+\|\pt_{1}\bfpsi(t)\|^2_{L^2}
			+\int_{0}^{t}\big(\|\nabla\pt_{1}\bfpsi(s)\|^2_{L^2}+\|\nabla\div \bfpsi(s)\|^2_{L^2}
			+\ep^{-4}\|G_1(s)\|^2_{L^2}\big)\md s\\
			\leq& C\big(\ep^{-2}\|\varphi_0\|^2_{H^1}+\|\pt_{1}\bfpsi_0\|^2_{L^2}+\ep^{-2}\|\varphi(t)\|^2_{L^2}\big)
			+C\int_{0}^{t}\ep^{-2}\|\pt_{1}\varphi(s)\|^2_{L^2}\md s\\
		&+C\int_{0}^{t}\big(\ep^{-4}\|\theta(s)\|^2_{H^1}
		+\ep^{-2}\|\bfpsi(s)\|^2_{H^1}
		+\|\pt_t\bfpsi(s)\|^2_{L^2}
		+\ep^{2}\|\pt_t\varphi(s)\|^2_{L^2}\big)\md s, 
		\quad\quad  \forall\, t\in[0,T], 
		\end{aligned}
	\end{equation}
	which, together with \eqref{eneA0}, \eqref{ene1} and \eqref{ene2}, implies \eqref{est-div}.
	
	The proof is completed.
\end{proof}

To derive the estimates for $\int_{0}^{T}\|\bfpsi(t)\|^2_{H^2}\md t$ and $\int_{0}^{T}\ep^{-4}\|\pt_{1}\varphi(s)\|^2_{L^2}\md t$, we rewrite $\eqref{eq}_2$ in the form of steady nonhomogeneous Stokes equations in the infinite channel $\Omega$, i.e.,
\begin{equation}
	\left.\begin{cases}
		\div \bfpsi=g, \quad\quad\quad\quad\quad\text{in}\,\,\Omega,\\
		-\mu\Delta \bfpsi+\ep^{-2}\nabla P(\rho,\fT)=\bfF,
		\quad\quad\quad\text{in}\,\,\Omega,\\
			\bfpsi=0 \quad\quad\mathrm{on}\,\,\pt\Omega,\\
		\bfpsi\rightarrow 0\quad\quad \mathrm{as} \,\,|\bfx|\rightarrow \infty.
	\end{cases}\label{Stokes2}
	\right.
\end{equation}
where 
\begin{equation}\label{def:gF}
	g=\div\bfpsi,\quad
	\bfF=-\rho\big(\pt_t\bfpsi+\bfu\nabla\bfpsi+\bfpsi\nabla\tbfu\big)
	+(\mu+\mup)\nabla\div\bfpsi. 
\end{equation}
We have the following lemma.
\begin{lemma}  \label{lemma4}
	Suppose that \eqref{chi} holds, and suppose that $\Re<\Re_0$ and $\ep<\vep_1$ as in \Cref{basic}. Then, there exists a positive constant $\vep_2$ depending only on $\Re$, $\Pr$, $\frac{\mup}{\mu}$ and $\gamma$, such that if $\ep\leq \vep_2$, then
	\begin{equation}	\label{ene6}
	\begin{aligned}
			\int_{0}^{t}\big(\|\bfpsi(s)\|^2_{H^2}+\ep^{-4}\|\pt_1\varphi\|^2_{L^2}\big)\md s
		\leq C\big(\ep^{-2}A_0(0)+A_1(0)+A_2(0)\big), \quad\quad \forall t\in [0,T],
	\end{aligned}
	\end{equation}
where the definitions of $A_0$, $A_1$ and $A_2$ are given in \eqref{def:A0123}.
\end{lemma}
\begin{proof}
We first note that
	\begin{equation}\notag
		\bfF=O(1)\big(|\pt_t\bfpsi|+|\nabla\bfpsi|+|\bfpsi|+|\nabla\div\bfpsi|\big).
	\end{equation}
Then, we obtain from \Cref{Stokes lemma} that
	\begin{equation}\label{est-Stokes}
		\begin{aligned}
			\mu\|\bfpsi\|^2_{H^2}+\ep^{-4}\|\nabla P(\rho,\fT)\|^2_{L^2}
			\leq& C(\|\bfF\|^2_{L^2}+\|g\|^2_{L^2})\\
			\leq& C\big(\|\nabla\div\bfpsi\|^2_{L^2}+\|\bfpsi\|^2_{H^1}
			      +\|\pt_t\bfpsi\|^2_{L^2}\big), 
		\end{aligned}
	\end{equation}
	
	Next, we observe that 
	\begin{equation}  \notag
		\pt_{1}P(\rho,\fT)=\fT\pt_{1}\rho+\rho\pt_{1}\fT
		=\fT\pt_{1}\varphi+\rho\pt_{1}\theta, 
	\end{equation}
	which, together with \eqref{ob-1}, gives that
	\begin{equation}\notag
		\begin{aligned}
		\int_{0}^{t}\ep^{-4}\|\pt_{1}\varphi(s)\|^2_{L^2} \md t
			\leq& C\int_{0}^{t}\ep^{-4}\big(\|\pt_{1} P(\rho,\fT)(s)\|^2_{L^2}  
			+\|\pt_{1} \theta(s)\|^2_{L^2}\big)\md t\\
			\leq&C\int_{0}^{t}\ep^{-4}\big(\|\nabla P(\rho,\fT)\|^2_{L^2} 
			+\|\theta(s)\|^2_{H^1}\big)\md t\\
			\leq& C\int_{0}^{t}\big(\|\nabla\div\bfpsi\|^2_{L^2}+\|\bfpsi\|^2_{H^1}
		+\|\pt_t\bfpsi\|^2_{L^2}+\ep^{-4}\|\theta\|^2_{H^1}\big)\md t.
		\end{aligned}
	\end{equation}
	
Applying the estimates \eqref{eneA0}, \eqref{ene1}, \eqref{ene2} and \eqref{est-div} to the above inequality yields 
\begin{equation}
		\int_{0}^{t}\ep^{-4}\|\pt_{1}\varphi(s)\|^2_{L^2} \md t
		\leq C\int_{0}^{t}\ep^{-2}\|\pt_{1}\varphi(s)\|^2_{L^2}\md s 
		+C\big(\ep^{-2}A_0(0)+A_1(0)+A_2(0)\big). 
\end{equation} 
This implies 
\begin{equation}\label{est-pt1rho}
		\int_{0}^{t}\ep^{-4}\|\pt_{1}\varphi(s)\|^2_{L^2} \md t
		\leq C\big(\ep^{-2}A_0(0)+A_1(0)+A_2(0)\big),
\end{equation}
provided that $\ep$ is sufficiently small.

Finally, \eqref{ene5} follows from \eqref{eneA0}, \eqref{ene1}, \eqref{ene2}, \eqref{est-pt1rho}, \eqref{est-div} and \eqref{est-Stokes}.
	
The proof is completed.
\end{proof}

To obtain the estimate for $\int_{0}^{T}\ep^{-2}\| \nabla^2\theta\|^2_{L^2}\md t$, we note that $\theta$ satisfies the elliptic system
\begin{equation}\label{eq:ellptictheta}
	\left.\begin{cases}
			-\kappa\Delta \theta=\tfH	\quad\quad\text{in}\,\,\Omega, \\
			\theta=0 \quad\quad\mathrm{on}\,\,\pt\Omega,\\
		\theta\rightarrow 0\quad\quad \mathrm{as} \,\,|\bfx|\rightarrow \infty.
	\end{cases}
	\right.
\end{equation}
where
\begin{equation}\label{def:H}
	\begin{aligned}
			\tfH=&-\rho\big(\pt_t\theta+ \bfu\cdot\nabla\theta+\bfpsi\cdot\nabla\tfT\big)
		-P(\rho,\fT) \div\bfpsi  \\
		&+\ep^{2}\big(2\mu|\Dpsi|^2+\mup(\div\bfpsi)^2+4\mu \Dtbfu:\Dpsi\big).
	\end{aligned}
\end{equation}

Then, we have the following lemma.
\begin{lemma} \label{lemma5}
	Suppose that \eqref{chi} holds, and suppose that $\Re<\Re_0$ and $\ep\leq\vep_2$ as in Lemma \ref{lemma4}. Then, it holds that
	\begin{equation}
		\ep^{-2}\int_{0}^{t}\|\nabla^2 \theta(s)\|^2_{L^2}\md s
		\leq C\big(\ep^{-2}A_0(0)+A_1(0)+A_2(0)\big),\quad\quad\quad \forall t\in[0,T], 	\label{ene5}
	\end{equation}
	where the definitions of $A_0$, $A_1$ and $A_2$ are given in \eqref{def:A0123}.
\end{lemma}
\begin{proof}
	Based on the classical elliptic theory \cite{ADN1959,ADN1964}, we obtain from \eqref{eq:ellptictheta}, \eqref{ob-1}, \eqref{ene2} and \eqref{ene6} that
	\begin{equation}\notag
		\begin{aligned}
			&\ep^{-2}\int_{0}^{t}\|\nabla^2 \theta(s)\|^2_{L^2}\md s\\
			\leq& C\ep^{-2}\int_{0}^{t}\big(\|\pt_t \theta(s)\|^2_{L^2}+\|\nabla\theta(s)\|^2_{L^2}
			+\|\nabla\tfT\|^2_{L^\infty}\|\bfpsi(s)\|^2_{L^2}
			+\ep^4(1+\|\nabla\bfpsi(s)\|^2_{L^4})\|\nabla \bfpsi(s)\|^2_{L^4}\big)\md s\\
	\leq& C\ep^{-2}\int_{0}^{t}\big(\|\pt_t \theta(s)\|^2_{L^2}+\|\theta(s)\|^2_{H^1}\big)\md s 
	 +C\ep^2\sup_{s\in[0,t]}\big(1+\|\bfpsi(s)\|^2_{H^2}\big)			\int_{0}^{t}\|\bfpsi(s)\|^2_{H^2}\md s\\
			\leq& C\big(\ep^{-2}A_0(0)+A_1(0)+A_2(0)\big),
			\quad\quad\quad \forall t\in[0,T],  
		\end{aligned}
	\end{equation}
	where we have used the fact from \eqref{focus} that $\sup_{s\in[0,t]}\ep^2\|\bfpsi(s)\|^2_{H^2}\leq \hat{N}\leq 1$.
	
	The proof is completed.
\end{proof}

In summary of \Cref{lemma1}--\Cref{lemma5}, we have the following lemma.
\begin{lemma}\label{lemmaH1}
	Suppose that \eqref{chi} holds, and suppose that $\Re<\Re_0$ and $\ep\leq\vep_2$ as in \Cref{lemma4}. Then, we have
	\begin{equation}\label{eneA12}
		A_1(t)+A_2(t) \leq C\big(\ep^{-2}A_0(0)+A_1(0)+A_2(0)\big),
		\quad\quad \forall t\in[0,T],  
	\end{equation}
	where the definitions of $A_0$, $A_1$ and $A_2$ are given in \eqref{def:A0123}.
\end{lemma}
\begin{proof}
	The estimate \eqref{eneA12} follows from \eqref{ene1}, \eqref{ene3}, \eqref{ene4}, \eqref{ene5} and \eqref{ene6}. This completes the proof.
\end{proof}

\subsection{Estimates for higher order derivatives.}
We are in a position to show the higher order \textit{a priori} estimates, which can be derived by using a similar argument as in the previous subsection. Since we consider only the case with small initial perturbation, we point out that with the \textit{a priori} estimates of $A_0$, $A_1$ and $A_2$ at hand, the estimates for those terms involving products of only lower order derivatives can be directly obtained by using \eqref{eneA0} and \eqref{eneA12}, provided that $\hat{N}$ is small enough (see \eqref{Nhat} for the definition of $\hat{N}$). We also refer the reader to \cite{MN1980,MN1983} for higher order estimates in the case with initial data around a steady state with zero velocity. In what follows we merely sketch the outline of the proof and mainly focus on the estimates for the terms involving higher order derivatives.

We first present the estimates for $\sup_{t\in [0,T]}\ep^2 A_3(t)$.
\begin{lemma} \label{lemmaH2temp}
	Suppose that \eqref{chi} holds, and suppose that $\Re<\Re_0$ and $\ep\leq\vep_2$ as in \Cref{lemma4}. Then, it holds that
	\begin{equation}\label{eneA3}
		\ep^{2}A_3(t)\leq C\big(\ep^{2}A_3(0)+A_2(0)+A_1(0)+\ep^{-2}A_0(0)\big),
		\quad  \forall t\in[0,T],
	\end{equation}
	where the definitions of $A_0$, $A_1$, $A_2$ and $A_3$ are given in \eqref{def:A0123}. 
\end{lemma}
\begin{proof}	
	It follows from the equation $\eqref{eq}_1$ and \eqref{ob-1} that
	\begin{equation}\notag
		\begin{aligned}
			\|\pt_{t}\varphi\|_{L^2}
			\leq& \big(\|\bfu\|_{L^\infty}\|\nabla\varphi\|_{L^2}
			+\|\rho\|_{L^\infty}\|\div\bfpsi\|_{L^2}
			+\|\trho\|_{L^\infty}\|\bfpsi\|_{L^2}\big)\\
			\leq & C\big(\|\nabla\varphi\|_{L^2}+\|\bfpsi\|_{H^1}\big),
		\end{aligned}
	\end{equation}
	which, together with \eqref{eneA0} and \eqref{eneA12}, implies
	\begin{equation}\label{ineq:pt1}
			\|\pt_{t}\varphi(t)\|^2_{L^2}\leq C\ep^2A_2(t)+C\big(A_1(t)+A_0(t)\big),
				\quad  \forall t\in[0,T],
	\end{equation}
	
	Taking the derivatives with respect to $t$ on the equations $\eqref{eq}_2$ and $\eqref{eq}_3$, and then multiplying the resulting equations by $\ep^2\pt_{t}\bfpsi$ and $\pt_{t}\theta$, respectively, we obtain, by using the method of integration by parts, that
	\begin{equation} \label{ineq:pt2}
		\begin{aligned}
			&\int\frac{1}{2}\big(\ep^2\rho|\pt_t\bfpsi|^2+\frac{1}{\gamma-1}\rho|\pt_t\theta|^2\big)(t)\md\bfx\\
			&+\ep^2\int_{0}^{t}\big(\mu\|\nabla\pt_t\bfpsi(s)\|^2_{L^2}+(\mu+\mup)\|\div\pt_t\bfpsi(s)\|^2_{L^2}\big)\md s
			+\int_{0}^{t}\kappa\|\nabla\pt_t\theta(s)\|^2_{L^2}\md s\\
			\leq&\int\frac{1}{2}\big(\ep^2\rho|\pt_t\bfpsi|^2
			+\frac{1}{\gamma-1}\rho|\pt_t\theta|^2\big)(0)\md\bfx + \frac{1}{2}\ep^2\int_{0}^{t}\mu\|\nabla\pt_{t}\bfpsi(s)\|^2_{L^2}\md s\\
			&+ \frac{1}{2}\int_{0}^{t}\kappa\|\nabla\pt_{t}\theta(s)\|^2_{L^2}\md s
			+C\int_{0}^{t}\big(\ep^{-2}\|\pt_{t}\varphi(s)\|^2_{L^2}
			+\ep^{-2}\|\pt_{t}\theta(s)\|^2_{L^2}+\|\pt_{t}\bfpsi(s)\|^2_{L^2}\big) \md s ,  
		\end{aligned}
	\end{equation}
	where we have used \eqref{ob-0}, \eqref{ob-1} and the facts from the Young inequality that
	\begin{equation}\notag
			\begin{aligned}
		\ep^2\int_{0}^{t}\int (\pt_{t}(\rho\bfpsi)\nabla\bfpsi)\cdot\pt_{t}\bfpsi\md\bfx\md s
			\leq& C\ep^2(\|\pt_{t}\varphi \|_{L^2}+\|\pt_{t}\bfpsi \|_{L^2})\|\nabla\bfpsi \|_{L^\infty}
			\|\pt_{t}\bfpsi \|_{L^2}\\
			\leq& C\ep(\|\pt_{t}\varphi \|^2_{L^2}+\|\pt_{t}\bfpsi \|^2_{L^2}),
		\end{aligned}
	\end{equation}
	and
	\begin{equation}\notag
		\begin{aligned}
			&\int \pt_{t}( P(\rho,\fT)\div\bfpsi)\pt_{t}\theta  \md \bfx\\
		=&	\int \pt_{t}P(\rho,\fT)\div\bfpsi\pt_{t}\theta  \md \bfx+
			\int P(\rho,\fT)\pt_{t}\div\bfpsi\pt_{t}\theta\md\bfx\\
		=&\int\big(\fT\pt_t\varphi+\rho\pt_{t}\theta\big)\div\bfpsi\pt_{t}\theta\md\bfx
		+\int P(\rho,\fT)\pt_{t}\div\bfpsi\pt_{t}\theta\md\bfx\\
		\leq& \big(\|\fT\|_{L^\infty}\|\pt_t\varphi\|_{L^2}
		+\|\rho\|_{L^\infty}\|\pt_t\theta\|_{L^2}\big)
		\|\div\bfpsi\|_{L^\infty}\|\pt_t\theta\|_{L^2}
		+\| P(\rho,\fT)\|_{L^\infty}\|\pt_{t}\theta\|_{L^2}\|\pt_{t}\div\bfpsi\|_{L^2} \\
		\leq&  	\frac{1}{8}\ep^2\mu\|\nabla\pt_{t}\bfpsi\|^2_{L^2}
		+C\ep^{-2}\big(\|\pt_{t}\varphi\|^2_{L^2}+\|\pt_{t}\theta\|^2_{L^2}\big),
		\end{aligned}
	\end{equation}
	and
		\begin{equation}\notag
		\begin{aligned}
			\int\nabla\pt_t P(\rho,\fT)\cdot\pt_t\bfpsi\md\bfx 
			=&-\int\pt_t P(\rho,\fT)\pt_t\div\bfpsi\md\bfx \\
=&-\int\big(\fT\pt_t\varphi+\rho\pt_{t}\theta\big)\pt_{t}\div\bfpsi\md\bfx \\
			\leq& \frac{1}{8}\ep^2\mu\|\nabla\pt_{t}\bfpsi\|^2_{L^2}
			+C\ep^{-2}\big(\|\pt_{t}\varphi\|^2_{L^2}+\|\pt_{t}\theta\|^2_{L^2}\big),
		\end{aligned}
	\end{equation}
	and 
	\begin{equation}\notag
		\begin{aligned}
			\int 4\ep^2\mu\big(\Dtbfu:\pt_{t}\Dpsi\pt_{t}\big)\theta \md\bfx 
			\leq& C\ep^2\mu\|\nabla\pt_{t}\bfpsi\|_{L^2}\|\pt_{t}\theta\|_{L^2}\\
				\leq&  	 \frac{1}{8}\ep^2\mu\|\nabla\pt_{t}\bfpsi\|^2_{L^2}
			+C\ep^{2}\|\pt_{t}\theta\|^2_{L^2}.
		\end{aligned}
	\end{equation}
	
	Noting that \eqref{eneA0} and \eqref{eneA12}, \eqref{ineq:pt2} leads to
	\begin{equation}\notag
			\begin{aligned}
		&\int\big(\ep^2\rho|\pt_t\bfpsi|^2+\rho|\pt_t\theta|^2\big)(t)\md\bfx\\
			\leq&C\int\big(\ep^2\rho|\pt_t\bfpsi|^2
		+\rho|\pt_t\theta|^2\big)(0)\md\bfx+C(A_2(t)+A_1(t)+A_0(t)), \quad  \forall t\in[0,T],
			\end{aligned}
	\end{equation}
	which, together with \eqref{ineq:pt1}, gives \eqref{eneA3}.
	
    The proof is completed.
\end{proof}

Next, we derive the $\ep$-weighted $H^2$-type estimates on $(\varphi,\bfpsi,\theta)$ as follows.
\begin{lemma} \label{lem:supH2}
	Suppose that \eqref{chi} holds, and suppose that $\Re<\Re_0$ and $\ep\leq\vep_2$ as in \Cref{lemma4}. Then, it holds that
	\begin{equation}\label{est:supH2}
		\ep^2\|\nabla^2\bfpsi(t)\|^2_{L^2}
		+\|\nabla^2\theta(t)\|^2_{L^2}
		\leq C\big(\ep^{2}A_3(0)+A_2(0)+A_1(0)+\ep^{-2}A_0(0)\big),\quad  \forall t\in[0,T], 
	\end{equation}
	where the definitions of $A_0$, $A_1$, $A_2$ and $A_3$ are given in \eqref{def:A0123}. 
\end{lemma}
\begin{proof}
	Rewrite the equations $\eqref{eq}_2$ as the Lam\'{e} system
\begin{equation}\label{elliptic}
	\left.\begin{cases}
		-\mu\Delta \bfpsi-(\mu+\mup)\nabla\div\bfpsi
		=-\ep^{-2}\nabla P(\rho,\fT)
		-\rho\big(\pt_t\bfpsi+\bfu\nabla \bfpsi+\bfpsi\nabla\tbfu\big),
		\quad\text{in}\,\,\Omega,\\
		\bfpsi=0 \quad\quad\mathrm{on}\,\,\pt\Omega,\\
		\bfpsi\rightarrow 0\quad\quad \mathrm{as} \,\,|\bfx|\rightarrow \infty.
	\end{cases}
	\right.
\end{equation}

Applying \Cref{Lame lemma} to the boundary value problem \eqref{elliptic}, we obtain, by using \eqref{ob-1} and \eqref{eneA3}, that
\begin{equation}\notag
	\begin{aligned}
		\ep^2\|\nabla^2\bfpsi(t)\|^2_{L^2}
		\leq& C\big(\ep^{-2}\|
		\nabla P(\rho,\fT)\|^2_{H^1}
		+\ep^2\|\bfpsi(t)\|^2_{H^1}+ \ep^2\|\pt_{t}\bfpsi(t)\|^2_{L^2}\big)\\
		\leq& C\big(\ep^{-2}\|\varphi(t)\|^2_{H^1}+\ep^{-2}\|\theta(t)\|^2_{H^1}
		+\ep^2\|\bfpsi(t)\|^2_{H^1}+ \ep^2\|\pt_{t}\bfpsi(t)\|^2_{L^2}\big)\\
		\leq&C\big(\ep^2A_3(t)+A_2(t)+A_1(t)+\ep^{-2}A_0(t)\big)\\ \leq&C\big(\ep^2A_3(0)+A_2(0)+A_1(0)+\ep^{-2}A_0(0)\big),\quad\quad  \forall t\in[0,T].
	\end{aligned}
\end{equation}
Similarly, applying the classical elliptic theory \cite{ADN1959,ADN1964} to the boundary value problem \eqref{eq:ellptictheta}, we have 
\begin{equation}\notag
	\begin{aligned}
		\|\nabla^2\theta(t)\|^2_{L^2}
		\leq& C\big(\|\bfpsi(t)\|^2_{H^1}+\|\pt_{t}\theta(t)\|^2_{L^2}
		+\|\nabla\theta(t)\|^2_{L^2}\big)\\
		\leq& C\big(\ep^2A_3(t)+A_2(t)+A_1(t)+\ep^{-2}A_0(t)\big)\\
		\leq& C\big(\ep^2A_3(0)+A_2(0)+A_1(0)+\ep^{-2}A_0(0)\big), \quad\quad  \forall t\in[0,T].
	\end{aligned}
\end{equation}
The proof is completed.
\end{proof}

\begin{lemma}\label{lem:nabladivpt1psi}
	Suppose that \eqref{chi} holds, and suppose that $\Re<\Re_0$ and $\ep\leq\vep_2$ as in \Cref{lemma4}. Let $\hat{N}$ be defined as in \eqref{Nhat}. Then, there exists a positive constant $N_1$ depending only on $\Re$, $\Pr$, $\frac{\mup}{\mu}$ and $\gamma$, such that if $\Nhat\leq N_1$, then we have
	\begin{equation} \label{est:nabladivpt1psi}
		\begin{aligned}
		\int_0^t\|\nabla\div\pt_{1}\bfpsi(s)\|^2_{L^2}\md s
		\leq& C\big(\|\nabla\pt_{1}\varphi(0)\|^2_{L^2}
		+\ep^2\|\pt^2_{1}\bfpsi(0)\|^2_{L^2}\big)\\
		&+C\big(\ep^2A_3(0)+A_2(0)+A_1(0)+\ep^{-2}A_0(0)\big), \quad\quad  \forall t\in[0,T],
	\end{aligned}
	\end{equation}
		where the definitions of $A_0$, $A_1$, $A_2$ and $A_3$ are given in \eqref{def:A0123}. 
\end{lemma}
\begin{proof}
		Similar to the proof of \Cref{lemma2}, we first apply the operator $\pt^2_{1}$ on equations \eqref{eqf1} and \eqref{eqf2} to get
	\begin{equation}
		\pt_t\pt^2_{1}\varphi+\bfu\cdot\nabla\pt^2_{1}\varphi+\div(\pt^2_{1}\bfpsi)
		=\pt^2_{1}f_1+M_1,   \label{eq1t}
	\end{equation}
	and
	\begin{equation}
		\begin{aligned}
			&\rho\big(\pt_t\pt^2_{1}\bfpsi+\bfu\cdot\nabla\pt^2_{1}\bfpsi+\pt^2_{1}\bfpsi\cdot\nabla\tbfu\big)
			+\ep^{-2}\nabla\big(\pt^2_{1}\varphi+\pt^2_{1}\theta\big)
			-\mu\Delta\pt^2_{1}\bfpsi-(\mu+\mup)\nabla\div\pt^2_{1}\bfpsi\\
			=&\ep^{-2}\nabla\pt^2_{1}f_2+\bfM_2+\bfM_3,\label{eq2t}
		\end{aligned}
	\end{equation}
	where
	\begin{equation}\notag
		\begin{aligned}
			&M_1=-\pt^2_{1}\bfpsi\cdot\nabla\varphi-2\pt_{1}\bfpsi\cdot\nabla\pt_{1}\varphi, \quad
			\bfM_2=-\pt^2_{1}\varphi(\pt_t\bfpsi+\bfu\cdot\nabla\bfpsi+\bfpsi\cdot\nabla\tbfu), \\
			&\bfM_3=-2\pt_{1}\varphi(\pt_t\pt_{1}\bfpsi+\bfu\cdot\nabla\pt_{1}\bfpsi+\pt_{1}\bfpsi\cdot\nabla\bfpsi). 
		\end{aligned}
	\end{equation}
	Adding \eqref{eq1t} multiplied by $\ep^{-2}\pt^2_{1}\varphi$ to \eqref{eq2t} multiplied by $\pt^2_{1}\bfpsi$, and then applying the method of integration by parts to the resulting equation, we obtain, by using \eqref{def:N}, \eqref{eneA0}, \eqref{eneA12}, \eqref{eneA3} and the Young inequality, that
	\begin{equation}\label{ineq:pt2divpsi1}
		\begin{aligned}
			&\|(\ep^{-1}\pt^2_{1}\varphi,\pt^2_{1}\bfpsi)(t)\|^2_{L^2}
			+\int_{0}^{t}\big(\|\nabla\pt^2_{1}\bfpsi(s)\|^2_{L^2}
			+\|\pt^2_{1}\div\bfpsi(s)\|^2_{L^2}\big)\md s\\
			\leq&C\|(\ep^{-1}\pt^2_{1}\varphi,\pt^2_{1}\bfpsi)(0)\|^2_{L^2}
			+C\big(\ep^{2}A_3(t)+A_2(t)+A_1(t)+\ep^{-2}A_0(t)\big),  \quad\quad\forall t\in[0,T],
		\end{aligned}
	\end{equation}
	provided that $\widehat{N}$ is small enough (but independent of $\ep$ and $T$). In deriving \eqref{ineq:pt2divpsi1}, to address the term $\int_{0}^{t}\int \bfM_2\cdot\pt^2_{1}\bfpsi \md\bfx \md s$, we have used the fact that
	\begin{equation}\notag
		\begin{aligned}
			\int_{0}^{t}\int\pt^2_{1}\varphi\pt_t\bfpsi\cdot\pt^2_{1}\bfpsi \md\bfx \md s
			\leq&\int_{0}^{t}\|\pt^2_{1}\varphi(s)\|_{L^2}\|\pt_t\bfpsi(s)\|_{L^4}\|\pt^2_{1}\bfpsi(s)\|_{L^4}\md s\\
			\leq&  \int_{0}^{t}\|\pt^2_{1}\varphi(s)\|^2_{L^2}\|\nabla\pt^2_{1}\bfpsi(s)\|^2_{L^2}\md s
			+C\int_{0}^{t}\|\nabla\pt_t\bfpsi(s)\|^2_{L^2}\md s,\\
			\leq& \big(\sup_{s\in [0,t]}\|\pt^2_{1}\varphi(s)\|^2_{L^2} \big)\int_{0}^{t}\|\nabla\pt^2_{1}\bfpsi(s)\|^2_{L^2}\md s
			+C\int_{0}^{t}\|\nabla\pt_t\bfpsi(s)\|^2_{L^2}\md s\\
			\leq& N(t)\int_{0}^{t}\|\nabla\pt^2_{1}\bfpsi(s)\|^2_{L^2}\md s
			+C\int_{0}^{t}\|\nabla\pt_t\bfpsi(s)\|^2_{L^2}\md s, 
		\end{aligned}
	\end{equation}
	and other terms are treated similarly.
	
Next, applying the operator $\pt_{1}$ on equations \eqref{eqnormal2} and \eqref{eqnormal1} gives
\begin{equation}\label{eq:pt1pt2psi}
	-(2\mu+\mup)\pt_1\pt^2_{2}\psi^2+\ep^{-2}\tfT\pt_{1}\pt_{2}\varphi=R_5,
\end{equation}
and
\begin{equation}\label{eq:pt1pt2varphi}
	\pt_t\pt_{1}\pt_{2}\varphi+\bfu\cdot\nabla\pt_{1}\pt_{2}\varphi+\pt^2_{1}\varphi
	+\rho\pt_{1}\pt^2_{2}\psi^2
	 = R_6,   
\end{equation}
where
\begin{equation}\label{def:R56}
		R_5=\pt_{1}R_3-\pt_{1}\varphi\pt_{2}\tfT,\quad
		R_6=\pt_{1}R_4-\pt_{1}\bfpsi\cdot\pt_{2}\varphi-\pt_{1}\varphi\pt^2_{2}\psi^2.
\end{equation}
In \eqref{def:R56}, the definitions of $R_3$ and $R_4$ are given in \eqref{def:R3} and \eqref{def:R4}, respectively.

Adding the equation \eqref{eq:pt1pt2psi} multiplied by $\rho^{-1}\pt_{1}\pt_{2}\varphi$ to the equation \eqref{eq:pt1pt2varphi} multiplied by $(2\mu+\mup)\pt_{1}\pt_{2}\varphi$, we obtain, by using the method of integration by parts, and \eqref{def:N}, \eqref{eneA0}, \eqref{eneA12}, \eqref{eneA3}, \eqref{ineq:pt2divpsi1} and the Young inequality, that
	\begin{equation}\label{ineq:pt2divpsi2}
		\begin{aligned}
			&\|\pt_{1}\pt_{2}\varphi(t)\|^2_{L^2}+\ep^{-2}\int_{0}^{t}\|\pt_{1}\pt_{2}\varphi(s)\|^2_{L^2}\md s\\
			\leq&C\|\pt_{1}\pt_{2}\varphi(0)\|^2_{L^2}+ C\ep^{2}\int_{0}^{t}\big(\|\nabla\pt^2_{1}\bfpsi(s)\|^2_{L^2}+\|\bfpsi(s)\|^2_{H^2}+\|\pt_{t}\bfpsi(s)\|^2_{H^1}\big)\md s\\
			&+C\ep^{-2}\int_{0}^{t}\big(\|\pt_1\varphi(s)\|^2_{L^2}
			+\|\theta(s)\|^2_{H^2}\big)\md s\\
			\leq&C\|\nabla\pt_{1}\varphi(0)\|^2_{L^2}+C\ep^2\|\pt^2_{1}\bfpsi(0)\|^2_{L^2}\\
			&+C\big(\ep^2A_3(t)+A_2(t)+A_1(t)+\ep^{-2}A_0(t)\big),   
			 \quad\quad  \forall t\in[0,T],
		\end{aligned}
	\end{equation}
	provided that $\widehat{N}$ is small enough. In deriving the above inequality, to obtain the estimate for the terms containing $\pt_{2}\varphi$, we have used the fact from \eqref{def:N} and the Young inequality that
	\begin{equation}\notag
			\begin{aligned}
&\big|\int_{0}^{t}\int \pt_{1}\bfpsi\cdot\pt_{2}\varphi\pt_{1}\pt_{2}\varphi \md\bfx \md s\big|\\
		\leq&\int_{0}^{t}\|\pt_{1}\pt_{2}\varphi(s)\|_{L^2}\|\pt_1\bfpsi(s)\|_{L^4}
		\|\pt_{2}\varphi(s)\|_{L^4}\md s\\
	 \leq&\int_{0}^{t}\|\pt_{1}\pt_{2}\varphi(s)\|_{L^2}\|\pt_1\bfpsi(s)\|_{H^1}
		\|\pt_{2}\varphi(s)\|_{H^1}\md s\\
		\leq & \eta\ep^2\int_{0}^{t}\|\pt_{1}\pt_{2}\varphi(s)\|^2_{L^2}\md s
		+C\frac{1}{\eta}\ep^{-2}\sup_{s\in[0,t]}\|\pt_{2}\varphi(s)\|^2_{H^1}
		\int_{0}^{t}\|\pt_1\bfpsi(s)\|^2_{H^1}\md s,\\
		\leq & \eta\ep^2\int_{0}^{t}\|\pt_{1}\pt_{2}\varphi(s)\|^2_{L^2}\md s
		+C\frac{1}{\eta}\ep^{-2}\hat{N}\int_{0}^{t}\|\pt_1\bfpsi(s)\|^2_{H^1}\md s,
		\quad\quad \forall \eta>0,\\
			\end{aligned}
	\end{equation}
	and other terms involving $\pt_{2}\varphi$ can be dealt with similarly.
	
	 Finally, applying the operator $\pt_{1}$ to the equation \eqref{eq:pt2divpsi}, we obtain from \eqref{ob-0} and \eqref{ob-1} that 
	\begin{equation}\notag
		\begin{aligned}
			(2\mu+\mup)\pt_{1}\pt_{2}\div\bfpsi
			=&O(\ep^{-2})(|\pt_{1}\pt_{2}\varphi|+|\pt_{1}\pt_{2}\theta|+|\pt_1\varphi|+|\nabla\theta|)\\
			&+O(1)(|\pt_{1}\pt_t\bfpsi|+|\nabla^2\bfpsi|+|\nabla\bfpsi|+|\bfpsi|+|\nabla\pt^2_{1}\bfpsi|),
		\end{aligned}
	\end{equation}
	which, together with \eqref{ineq:pt2divpsi1}, \eqref{ineq:pt2divpsi2}, \eqref{eneA0}, \eqref{eneA12} and \eqref{eneA3}, leads to \eqref{est:nabladivpt1psi}.
	
	The proof is completed.
\end{proof}

\begin{lemma}\label{lem:H2tan}
	Suppose that \eqref{chi} holds, and suppose that $\Re<\Re_0$ and $\ep\leq\vep_2$ as in \Cref{lemma4}. Let $\hat{N}$ be defined as in \eqref{Nhat} and suppose that $\hat{N}\leq N_1$ as in \Cref{lem:nabladivpt1psi}. Then, it holds that
\begin{equation}\label{est:H2tan}
	\begin{aligned}
		&\int_{0}^{t}\big(\|\pt_{1}\bfpsi(s)\|^2_{H^2}
		+\|\pt_{2}^{3}\psi^1(s)\|_{L^2}^2
		+\ep^{-4}\|\nabla\pt_{1}\varphi\|^2_{L^2}\big)\md s\\
		\leq &C\int_{0}^{t}\big(\|\nabla\div\pt_{1}\bfpsi(s)\|^2_{L^2}
		+\|\pt_{t}\bfpsi(s)\|^2_{H^1}
		+\|\bfpsi(s)\|^2_{H^2}+\|G_1(s)\|^2_{L^2}\big)\md s\\
		&+C\int_{0}^{t}\ep^{-4}\big(\|\pt_{1}\varphi\|^2_{L^2}
		+\|\theta\|^2_{H^2}\big)\md s, 
	 \quad\quad \forall t\in[0,T].
	\end{aligned}
\end{equation} 
\end{lemma}
\begin{proof}
	We first consider the elliptic system obtained from \eqref{Stokes2},
	\begin{equation}\label{Stokestan}
		\left.\begin{cases}
			\div (\pt_{1}\bfpsi)=\pt_{1}g, \quad\quad\quad\quad\quad\text{in}\,\,\Omega,\\
			-\mu\Delta 	(\pt_{1}\bfpsi)+\ep^{-2}\nabla (\pt_{1}P(\rho,\fT))=\pt_{1}\bfF,
			\quad\quad\quad\text{in}\,\,\Omega,\\
			\pt_{1}\bfpsi=0 \quad\quad\mathrm{on}\,\,\pt\Omega,\\
			\pt_{1}\bfpsi\rightarrow 0\quad\quad \mathrm{as} \,\,|\bfx|\rightarrow \infty.
		\end{cases}
		\right.
	\end{equation}
	where the definitions of $g$ and $\bfF$ can be found in \eqref{def:gF}. 
	
	It follows from \Cref{Stokes lemma} that
	\begin{equation}\label{ineq:H2pt1psi}
		\begin{aligned}
			&\int_{0}^{t}\big(\|\pt_{1}\bfpsi(s)\|^2_{H^2}+\ep^{-4}\|\nabla\pt_{1}P(s)\|^2_{L^2}\big)\md s\\
			\leq & C\int_{0}^{t}\big(\|\pt_{1}g(s)\|^2_{L^2}+\|\pt_{1}\bfF(s)\|^2_{L^2}\big)\md s\\
			\leq &C\int_{0}^{t}\big(\|\nabla\div\pt_{1}\bfpsi(s)\|^2_{L^2}+\|\nabla\pt_{t}\bfpsi(s)\|^2_{L^2}
			+\|\bfpsi(s)\|^2_{H^2}\big)\md s,
		\end{aligned}
	\end{equation}
	
	Next, note the fact that
	\begin{equation}\notag
		\nabla\pt_{1}P(\rho,\fT)
		= \fT\nabla\pt_{1}\varphi+ \pt_{1}\varphi\nabla\fT +\varphi\nabla\pt_{1}\theta+\pt_{1}\theta\nabla\varphi.
	\end{equation}
	Thus, it follows from \eqref{ob-0} and \eqref{ob-1} that
	\begin{equation}\label{ineq:nablapt1phi}
		\int_{0}^{t}\|\nabla\pt_{1}\varphi(s)\|^2_{L^2}\md s
		\leq C\int_{0}^{t}\big(\|\nabla\pt_{1}P(s)\|^2_{L^2}
		+\|\pt_{1}\varphi(s)\|^2_{L^2}+\|\theta(s)\|^2_{H^2}\big)\md s.
	\end{equation}
	
	Then, it is obtained from the equation $\eqref{eq}_2$ that
	\begin{equation}\notag
		\begin{aligned}
			\mu\pt_{2}^{3}\psi^1
			=&-(\mu+\mu^\prime)\pt_{2}^2\pt_{1}\psi^2
			-(2\mu+\mu^\prime)\pt_{1}^{2}\pt_{2}\psi^1
	+\ep^{-2}\pt_{2}\big(\trho\pt_{1}\theta+\tfT\pt_{1}\varphi\big)\\
			&+\pt_{2}\big(\rho\pt_t\psi^1+\rho\bfu\cdot\nabla \psi^1+\rho\bfpsi\cdot\nabla\tu^1\big).
		\end{aligned}
	\end{equation}
	This, together with the fact that $\pt_{2}\varphi = \frac{1}{\tfT}G_1-\frac{\pt_2\tfT}{\tfT}\varphi$, implies
	\begin{equation}\label{ineq:pt222psi1}
		\begin{aligned}
	\int_{0}^{t}\|\pt_{2}^{3}\psi^1(s)\|_{L^2}^2\md s
	\leq	&C\int_{0}^{t}\big(\|\nabla^2\pt_{1}\bfpsi(s)\|^2_{L^2}
	+\|\pt_{t}\bfpsi(s)\|^2_{H^1}
	+\|\bfpsi(s)\|^2_{H^2} +\|G_1(s)\|^2_{L^2}\big)\md s\\
	+&C\ep^{-4}\int_{0}^{t}\big(\|\nabla\pt_{1}\varphi(s)\|^2_{L^2}
	+\|\theta(s)\|^2_{H^2}\big)\md s.
		\end{aligned}
	\end{equation}
	
	Finally, combining \eqref{ineq:nablapt1phi}, \eqref{ineq:H2pt1psi} and \eqref{ineq:pt222psi1} gives \eqref{est:H2tan}.
	
	The proof is completed.
	\end{proof}
	\begin{lemma}\label{lem:pt2G1pt222psi2}
		Suppose that \eqref{chi} holds, and suppose that $\Re<\Re_0$ and $\ep\leq\vep_2$ as in \Cref{lemma4}. Let $\hat{N}$ be defined as in \eqref{Nhat}. Then, there exists a positive constant $N_2$ depending only on $\Re$, $\Pr$, $\frac{\mup}{\mu}$ and $\gamma$, such that if $\Nhat\leq N_2$, then we have
		\begin{equation}\label{est:pt2G1pt222psi2}
			\begin{aligned}
			&\|\pt^2_{2}\varphi(t)\|^2_{L^2}
			+\ep^{-2}\int_{0}^{t}\|\pt_{2}G_1(s)\|^2_{L^2}\md s
	+\ep^2\int_{0}^{t}\|\pt^2_{3}\psi^2(s)\|^2_{L^2}\md s	\\ 
			\leq&C\big(\|\pt_{2}\varphi(t)\|^2_{H^1}
			+\|\varphi(0)\|^2_{H^2}\big)
			+C\ep^{2}\int_{0}^{t}\big(\|\nabla^2\pt_{1}\bfpsi(s)\|^2_{L^2}+\|\bfpsi(s)\|^2_{H^2}
			+\|\pt_{t}\bfpsi(s)\|^2_{H^1}\big)\md s\\
			&+C\ep^{-2}\int_{0}^{t}\big(\|\pt_{1}\pt_{2}\varphi(s)\|^2_{L^2}+\|G_1\|^2_{L^2}
			+\|\pt_{1}\varphi(s)\|^2_{L^2}
			+\|\pt_{t}\varphi(s)\|^2_{L^2}
			+\|\theta(s)\|^2_{H^2}\big)\md s,
				\end{aligned}
		\end{equation}
		for any $t\in[0,T]$.
	\end{lemma}
	\begin{proof}
    Noting that $\pt_{2}\varphi = \frac{1}{\tfT}G_1-\frac{\pt_2\tfT}{\tfT}\varphi$, the equation \eqref{eqnormal1} is rewritten as
    \begin{equation}\label{eq:G1}
    	\frac{1}{\tfT}\pt_{t}G_1+\frac{1}{\tfT}\bfu\cdot\nabla G_1
    	+\pt_{1}\varphi+\rho\pt_{2}^2\psi^2
    	=R_7,
    \end{equation}
    where
    \begin{equation}\notag
    	R_7=R_4 +\frac{\pt_{2}\tfT}{\tfT}\pt_{t}\varphi-\frac{\pt_{2}\tfT}{\tfT}\psi^2 G_1
    	-\frac{1}{\tfT}\bfu\cdot\nabla G_1
    	+\bfu\cdot\nabla(\frac{\pt_{2}\tfT}{\tfT})\varphi
    	+\frac{\pt_{2}\tfT}{\tfT}\bfu\cdot\nabla \varphi,
    \end{equation}
    and the definition of $R_4$ is given in \eqref{def:R4}.
    
	Applying the operator $\pt_{2}$ on equations \eqref{eq:G1} and \eqref{eqnormal2}, we obtain
	\begin{equation}\label{eq:ptx2G1}
			\frac{1}{\tfT}\pt_{t}\pt_{2}G_1+\frac{1}{\tfT}\bfu\cdot\nabla\pt_{2} G_1
		+\pt_{1}\pt_{2}\varphi+\rho\pt_{2}^3\psi^2
		=\pt_{2}R_7-\pt_{2}(\frac{\bfu}{\tfT})\cdot\nabla G_1-\pt_{2}\rho\pt_{2}^2\psi^2,  
	\end{equation}
	and
		\begin{equation}\label{eq:ptx22psi2}
		-(2\mu+\mup)\pt^2_{3}\psi^2
		+\ep^{-2}\pt_{2}G_1=\pt_{2}R_3.
	\end{equation}
	where the definition of $R_3$ is given in \eqref{def:R3}.
	
	Noting the facts from \eqref{def:G} and \eqref{def:Couette} that 
	\begin{equation}\notag
		\pt_{1}G_1=\tfT\pt_1\pt_2\varphi+\pt_1\pt_2\tfT,\quad
		 \pt_{2}\big(\frac{u^1}{\tfT})=\pt_{2}(\frac{x_2+\psi^1}{\tfT}\big),\quad
		 \pt_{2}\big(\frac{u^2}{\tfT})=\pt_{2}(\frac{\psi^2}{\tfT}\big),
	\end{equation}
	we obtain from \eqref{ob-1} that 
	\begin{equation}\label{ineq:nablaG1}
		\begin{aligned}
				\|\pt_{2}(\frac{\bfu}{\tfT})\cdot\nabla G_1\|_{L^2}
			\leq&  \|\pt_{2}(\frac{1+\psi^1}{\tfT})\|_{L^\infty}\|\pt_{1} G_1\|_{L^2} 
			 + \|\pt_{2}(\frac{\psi^2}{\tfT})\|_{L^\infty}\|\pt_{2} G_1\|_{L^2} \\
			\leq& C(\hat{N}+\ep^{-1}\hat{N})\|\pt_{1} G_1\|_{L^2}  + C\ep^{-1}\hat{N}\|\pt_{2} G_1\|_{L^2} \\
			\leq& C\ep^{-1}\big(\|\pt_{1}\pt_{2}\varphi\|_{L^2} 
			+ \|\pt_{1}\varphi\|_{L^2}\big)  + C\ep^{-1}\hat{N}\|\pt_{2} G_1\|_{L^2}.
		\end{aligned}
	\end{equation}
	Thus, in analogy with derivation of the estimate \eqref{div3}, adding the equation \eqref{eq:ptx2G1} multiplied by $(2\mu+\mup)\pt_{2}G_1$ to the equation \eqref{eq:ptx22psi2} multiplied by $\rho^{-1}\pt_{2}G_1$, and then integrating the resulting equation over $\Omega\times[0,t]$, we obtain, by using \eqref{def:N}, \eqref{ineq:nablaG1} and the Young inequality, that
	\begin{equation} \label{ineq:pt2G1}
		\begin{aligned}
			&\|\pt_{2}G_1(t)\|^2_{L^2}+\ep^{-2}\int_{0}^{t}\|\pt_{2}G_1(s)\|^2_{L^2}\md s\\
			\leq&C\|\varphi(0)\|^2_{H^2}+ C\ep^{2}\int_{0}^{t}\big(\|\nabla^2\pt_{1}\bfpsi(s)\|^2_{L^2}+\|\bfpsi(s)\|^2_{H^2}
			+\|\pt_{t}\bfpsi(s)\|^2_{H^1}\big)\md s\\
			&+C\ep^{-2}\int_{0}^{t}\big(\|\pt_{1}\pt_{2}\varphi(s)\|^2_{L^2}
			+\|\pt_{1}\varphi(s)\|^2_{L^2}+\|\pt_{t}\varphi(s)\|^2_{L^2}
			+\|\theta(s)\|^2_{H^2}\big)\md s,
		\end{aligned}
	\end{equation}
	provided that $\widehat{N}$ is small enough. 
	
     Next, the equation \eqref{eq:ptx22psi2} gives
     \begin{equation}\notag
     	\begin{aligned}
     		\pt^2_{3}\psi^2
     		=&O(\ep^{-2})(|\pt_{2}G_1|+|\pt^2_{2}\theta|+|G_1|+|\nabla\theta|)\\
     		&+O(1)(|\pt_{2}\pt_t\bfpsi|+|\nabla^2\bfpsi|+|\nabla\bfpsi|+|\bfpsi|
     		+|\nabla^2\pt_{1}\bfpsi|),
     	\end{aligned}
     \end{equation}
     This leads to
     \begin{equation}\label{ineq:pt222psi2}
     	\begin{aligned}
     	\int_{0}^{t}\ep^2\|\pt^2_{3}\psi^2(s)\|^2_{L^2}\md s
     		\leq & C\ep^2\int_{0}^{t}\big(\|\nabla^2\pt_{1}\bfpsi(s)\|^2_{L^2}+\|\bfpsi(s)\|^2_{H^2}
     		+\|\pt_{t}\bfpsi(s)\|^2_{H^1}\big)\md s\\
     		&+C\ep^{-2}\int_{0}^{t}\big(\|\pt_{2}G_1\|^2_{L^2}+\|G_1\|^2_{L^2}
     		+\|\theta(s)\|^2_{H^2}\big)\md s.
     	\end{aligned}
     \end{equation}

     In addition, it is obtained from \eqref{def:G} that
     \begin{equation}\notag
     	\begin{aligned}
     		\pt_{2}^2\varphi
     		=&\frac{1}{\tfT}\pt_2 G_1 	+\frac{\pt_{2}\tfT}{\tfT}G_1-\pt_{2}\tfT\pt_2\varphi
     		-\pt_{2}\big(\frac{\pt_{2}\tfT}{\tfT}\big)\varphi\\
     		=&O(\pt_{2}G_1)+O(\pt_{2}\varphi)+O(\varphi),
     	\end{aligned}
     \end{equation}
     which implies 
     \begin{equation}\label{ineq:pt22phi}
     	\|\pt^2_{2}\varphi(t)\|_{L^2}
     	\leq C\big(	\|\pt_{2}G_1(t)\|^2_{L^2}
     	+\|\varphi(t)\|^2_{H^1}\big), \quad\quad\forall t\in[0,T].
     \end{equation}
     
  Finally, combining \eqref{ineq:pt2G1},\eqref{ineq:pt222psi2} and \eqref{ineq:pt22phi} gives \eqref{est:pt2G1pt222psi2}.
  
  The proof is completed.
\end{proof}

	\begin{lemma}\label{lem:nabla3theta}
	Suppose that \eqref{chi} holds, and suppose that $\Re<\Re_0$ and $\ep\leq\vep_2$ as in \Cref{lemma4}. Let $\hat{N}$ be defined as in \eqref{Nhat} and suppose that $\hat{N}\leq N_2$ as in \Cref{lem:pt2G1pt222psi2}. Then, it holds that
	\begin{equation}\label{est:nabla3theta}
		\begin{aligned}
			\int_{0}^{t}\|\nabla^{3}\theta(s)\|^2_{L^2}\md s	
			\leq &   C\ep^{-4}\int_0^t\big(\|\pt_{2}\theta(s)\|^2_{L^2}
			+\|\pt_{1}\varphi(s)\|^2_{L^2}
			+\|G_1(s)\|^2_{L^2}\big)\md s\\
			&+C\int_0^t\big(
			\|\pt_{t}\theta(s)\|^2_{L^2}
			+\|\bfpsi(s)\|^2_{H^2}+\|\theta(s)\|^2_{H^2}		
			\big)\md s,\quad\forall t\in[0,T].\\
		\end{aligned}
	\end{equation}
\end{lemma}
\begin{proof}
	Applying the operator $\pt_{1}$ on the elliptic system \eqref{eq:ellptictheta} leads to
\begin{equation}\label{eq:pt1theta}
	\left.\begin{cases}
		-\kappa\Delta \pt_{1}\theta=\pt_{1}\tfH
			\quad\quad\text{in}\,\,\Omega, \\
		 \pt_{1}\theta=0 \quad\quad\mathrm{on}\,\,\pt\Omega,\\
		 \pt_{1}\theta\rightarrow 0\quad\quad \mathrm{as} \,\,|\bfx|\rightarrow \infty,
	\end{cases}
	\right.
\end{equation}
where $\tfH$ is defined as in \eqref{def:H}. Based on the above system, it is obtained from the classical elliptic theory \cite{ADN1959,ADN1964} that
\begin{equation}\label{ineq:nabla2pt1theta}
	\begin{aligned}
			\int_0^t\|\nabla^2\pt_{1}\theta(s)\|^2_{L^2}\md s
			\leq & 	C\int_0^t\|\pt_{1}\tfH(s)\|^2_{L^2}\md s\\
			\leq &   C\ep^{-4}\int_0^t\big(\|\pt_{1}\varphi(s)\|^2_{L^2}
			+\|\pt_{1}\theta(s)\|^2_{L^2}\big)\md s\\
			&+C\int_0^t\big(
			\|\pt_{t}\theta(s)\|^2_{L^2}
			+\|\bfpsi(s)\|^2_{H^2}+\|\theta(s)\|^2_{H^2}		
			\big)\md s.\\
	\end{aligned}
\end{equation}

Then, applying the operator $\pt_{2}$ on the elliptic system \eqref{eq:ellptictheta} gives the equation
\begin{equation}\notag
	\kappa\pt_{2}^3\theta=-\kappa\pt_{1}^2\pt_{2}\theta-\pt_{2}\tfH,
\end{equation}
which implies
\begin{equation}\label{ineq:pt222theta}
	\begin{aligned}
	\int_0^t\|\pt_{2}^3\theta(s)\|^2_{L^2}\md s
	\leq & 	C\int_0^t\big(\|\nabla^2\pt_{1}\theta(s)\|^2_{L^2}
	+\|\pt_{2}\tfH(s)\|^2_{L^2}\big)\md s\\
	\leq &   C\ep^{-4}\int_0^t\big(\|\pt_{2}\theta(s)\|^2_{L^2}
	+\|G_1(s)\|^2_{L^2}\big)\md s\\
	&+C\int_0^t\big(
	\|\pt_{t}\theta(s)\|^2_{L^2}
	+\|\bfpsi(s)\|^2_{H^2}+\|\theta(s)\|^2_{H^2}		
	\big)\md s.\\
\end{aligned}
\end{equation}

Finally, the estimate \eqref{est:nabla3theta}  follows from \eqref{ineq:nabla2pt1theta} and \eqref{ineq:pt222theta}.

The proof is completed.
\end{proof}

Here we summarize the results from \Cref{lem:supH2,lem:nabladivpt1psi,lem:H2tan,lem:pt2G1pt222psi2,lem:nabla3theta} as the following lemma.
	\begin{lemma}\label{lemmaH2}
	Suppose that \eqref{chi} holds, and suppose that $\Re<\Re_0$ and $\ep\leq\vep_2$ as in \Cref{lemma4}. Let $\hat{N}$ be defined as in \eqref{Nhat} and suppose that $\hat{N}\leq N_2$ as in \Cref{lem:pt2G1pt222psi2}. Then, it holds that
	\begin{equation}\label{eneA4}
		\ep^{2}A_4(t)\leq C\big(\ep^2A_4(0)+\ep^2A_3(0)+A_2(0)+A_1(0)+\ep^{-2}A_0(0)\big), \quad\quad \forall t\in[0,T]. 
	\end{equation}
	where the definitions of $A_0$, $A_1$, $A_2$, $A_3$ and $A_4$ can be found in \eqref{def:A0123} and \eqref{def:A45}. 
\end{lemma}
\begin{proof}
	Based on \eqref{eneA0}, \eqref{eneA12} and \eqref{eneA3}, the estimate \eqref{eneA4} directly follows from \eqref{est:supH2},  \eqref{est:nabladivpt1psi}, \eqref{est:H2tan}, \eqref{est:pt2G1pt222psi2} and \eqref{est:nabla3theta}.
	
	The proof is completed.
\end{proof}

Finally, we state $\ep$-weighted $H^3$-type estimates on $(\varphi,\bfpsi,\theta)$ as follows.
\begin{lemma} \label{lemmaH3}
	Suppose that \eqref{chi} holds, and suppose that $\Re<\Re_0$ as in \Cref{lemmaH2}. Let $\hat{N}$ be defined as in \eqref{Nhat}. Then, there exists positive constants $\vep_3$ and $N_3$ depending only on $\Re$, $\Pr$, $\frac{\mup}{\mu}$ and $\gamma$, such that if $\ep\leq \vep_3$ and if $\Nhat\leq N_3$, then it holds that
	\begin{equation}
		\ep^{4}A_5(t)\leq C\big(\ep^{4}A_5(0)+\ep^2A_4(0)+\ep^2A_3(0)+A_2(0)+A_1(0)+\ep^{-2}A_0(0)\big), \quad\forall t\in[0,T], \label{eneA5}
	\end{equation}
	where the definitions of $A_0$, $A_1$, $A_2$, $A_3$, $A_4$ and $A_5$ can be found in \eqref{def:A0123} and \eqref{def:A45}. 
\end{lemma}
\begin{proof}
	This lemma can be proved by using a similar argument as in the proof of \Cref{lemmaH2temp,lemmaH2}. The details are omitted here.
\end{proof}

\section{Proof of the main theorems}
In this section we prove \Cref{main1,main2}. 

\subsection{Proof of \Cref{main1}}
To prove \Cref{main1}, we will follow standard arguments for problems with small data as in \cite{MN1980,MN1983}. Thus, we only give a sketch of proof as follows.

{\bf \hspace{-1em}Proof of \Cref{main1}.}
Let $\Re^\prime$, $\vep^\prime$, $N^\prime$ and $\hat{C}$ be the same as in \Cref{priori}. By \Cref{local}, there exist a time $T_*>0$ and a unique strong solution $(\varphi, \bfpsi,\theta)$ to the initial-boundary value problem \eqref{eq}--\eqref{eq5} on $(0,T_*)\times\Omega$ such that \eqref{local0} holds. 

Let $N(t)$ be defined by \eqref{def:N}, and suppose that 
\begin{equation}
	N(0)\leq \max\{\frac{1}{4},\frac{1}{16}\hat{C}^{-1}\}N^\prime.  \label{prove2}
\end{equation}
Then, due to \eqref{local0} there exists a time $t_1\in (0,T_*]$ such that
\begin{equation}
	\sup_{t\in(0,t_1)}N(t)\leq 2N(0) \leq \frac{1}{2}N^\prime. \label{prove1}
\end{equation}
Thus, it follows from \eqref{prove1} and \Cref{priori} that 
\begin{equation}
	\sup_{t\in(0,t_1)}N(t)\leq \hat{C}N(0),
\end{equation}
which, together with \eqref{prove2}, leads to 
\begin{equation}
	\sup_{t\in(0,t_1)}N(t)\leq \frac{1}{4}N^\prime.
\end{equation}

Next, we can solve the problem \eqref{eq}--\eqref{eq5} in $t\geq t_1$ with initial data $(\varphi(t_1),\bfpsi(t_1),\theta(t_1))$ again, and by uniqueness we can extend the solution $(\varphi,\bfpsi,\theta)$ to $[0,2t_1]$. Therefore, we can continue the above argument and the same process for $0\leq t\leq nt_1$, $n=2,3,4,\cdots$ and finally obtain a global unique strong solution $(\varphi,\bfpsi,\theta)$ satisfying \eqref{priori} for any $t>0$. Let $(\rho,u,\fT)=(\trho+\varphi,\tbfu+\bfpsi,\tfT+\theta)$. It can be seen that $(\rho,u,\fT)$ is indeed a global unique strong solution to the original problem \eqref{eq:2DCNS}--\eqref{initialcon} such that \eqref{regularity} and \eqref{uniform} hold. 

Finally, the large time behavior \eqref{largetime} can be shown by using the Sobolev embedding theorem and the fact from \eqref{priori} that $(\bfpsi,\theta)\in H^1\big([0,\infty);H^2(\Omega)\big)$.

The proof is completed.                                   \,\,\,$\hspace{25em}\Box$

\subsection{Proof of \Cref{main2}}
Now, we are ready to prove \Cref{main2}.

{\bf \hspace{-1em}Proof of \Cref{main2}.}
The condition \eqref{same} implies that \eqref{chi} holds. Therefore, \Cref{main1} guarantees the global existence of strong solution $(\rho^\ep,u^\ep,\fT^\ep)$ to \eqref{eq:2DCNS}--\eqref{initialcon}, which satisfies \eqref{regularity} and \eqref{uniform}, i.e., it holds that
\begin{equation}
	\begin{aligned}
		\|\rho^\ep-\trho\|_{L^2}=O(\ep^2),
		\quad \|\bfu^\ep-\tbfu\|_{L^2}=O(\ep),
		\quad \|\fT^\ep-\tfT\|_{L^2}=O(\ep^2).
	\end{aligned}  \label{converge1}
\end{equation}
Moreover, it is observed from \eqref{def:Couette} and \eqref{same} that 
\begin{equation}
	\begin{aligned}
		\|\trho-1\|_{L^2}=O(\ep^2),
		\quad \|\tfT-1\|_{L^2}=O(\ep^2).\label{converge2}
	\end{aligned} 
\end{equation}
Combining \eqref{converge1} and \eqref{converge2} gives \eqref{converge0}.

The proof is completed.
\,\,\,$\hspace{25em}\Box$\\

%\begin{appendices}
%appendices.
%\end{appendices}

\noindent {\bf Acknowledgments:}
The authors are supported by the NSFC (Grants 12131007).

%\bibliographystyle{plain}
%\bibliographystyle{siam}
%\bibliography{PDEref}

\end{document}